\newif\ifpaper
\paperfalse

\ifpaper

	\documentclass{mcom-l}

	\usepackage{amsmath,amssymb,amsfonts,amsthm,thmtools}

	\usepackage[pdftex]{graphicx} 

	\copyrightinfo{2009}{American Mathematical Society}

	\newtheorem{theorem}{Theorem}[section]
	\newtheorem{lemma}[theorem]{Lemma}

	\theoremstyle{definition}
	\newtheorem{definition}[theorem]{Definition}
	\newtheorem{example}[theorem]{Example}

	\theoremstyle{remark}
	\newtheorem{remark}[theorem]{Remark}
	\newtheorem{corollary}[theorem]{Corollary}
	\newtheorem{lp}{\normalfont{LP}}

	\numberwithin{equation}{section}

\else

	\documentclass[11pt,a4paper,oneside]{article}

	\usepackage{amsmath,amssymb,amsfonts,amsthm,thmtools}
	\usepackage[pdftex]{graphicx} 
	\usepackage[hmargin=1.0in,bottom=1.8in,footskip=1.0in]{geometry}
	\usepackage[usenames]{color} 
	\usepackage[sc]{mathpazo} 
	\usepackage[textsize=footnotesize]{todonotes}
	\usepackage{sectsty}
	\usepackage[
		style=numeric-comp,
		hyperref=true,
		backend=bibtex,
		doi=false,
		url=false,
		isbn=false,
		backref=false,
		firstinits=true,
		sorting=nyt,
		maxcitenames=3,
		maxbibnames=100,
		block=none]{biblatex}
	\definecolor{lightgray}{gray}{0.5}
	\definecolor{darkblue}{rgb}{0.,0.,0.5}
	\definecolor{darkred}{rgb}{0.6,0.,0.}
	\usepackage[linktocpage=true,colorlinks=true,linkcolor=blue,citecolor=red,urlcolor=blue]{hyperref}

	\renewcommand{\newline}{\vspace{10pt}}

	\declaretheorem[style=plain,numberwithin=section]{theorem}
	\declaretheorem[style=plain,numberwithin=section]{lemma}
	
	\declaretheorem[style=plain,numberwithin=section]{corollary}
	\declaretheorem[style=definition,numberwithin=section]{definition}
	\declaretheorem[style=definition,numberwithin=section]{example}
	\declaretheorem[style=definition,numberwithin=section]{remark}
	\declaretheorem[style=definition,name=LP]{lp}
	\numberwithin{equation}{section}
	\numberwithin{table}{section}
	\numberwithin{figure}{section}

	\sectionfont{\raggedright}

	\DeclareNameAlias{default}{last-first}
	\newbibmacro*{string+doiurlisbn}[1]{
		\iffieldundef{doi}{
			\iffieldundef{url}{
				\iffieldundef{isbn}{
					\iffieldundef{issn}{#1}
						{\href{http://books.google.com/books?vid=ISSN\thefield{issn}}{#1}}}
						{\href{http://books.google.com/books?vid=ISBN\thefield{isbn}}{#1}}}
						{\href{\thefield{url}}{#1}}}
						{\href{http://dx.doi.org/\thefield{doi}}{#1}}
	}
	\DeclareFieldFormat*{title}{\usebibmacro{string+doiurlisbn}{#1\isdot}}
	
	\renewbibmacro{in:}{}
	\AtEveryBibitem{\clearfield{number}}

\fi

	\usepackage{textcomp}
	\usepackage[T3,T1]{fontenc}
	\usepackage{bm}
	\usepackage{enumerate}
	\usepackage{paralist}
	\usepackage{enumitem}
	\usepackage{url}
	\usepackage{array}
	\usepackage{arydshln}
	\usepackage{subfigure}
	\usepackage{caption} 
	\usepackage{verbatim}

	\graphicspath{{figures/}}

	\newcommand{\rk}{Runge--Kutta method }
	\newcommand{\ark}{additive \rk}
	
	\newcommand{\lmm}{linear multistep method }
	\newcommand{\almm}{additive LMM }
	\newcommand{\plmm}{perturbed LMM }
	
	\renewcommand{\u}{\bm{u}}
	
	\newcommand{\U}{\bm{U}}
	\newcommand{\f}{\bm{f}}
	\newcommand{\F}{\bm{F}}
	
	\newcommand{\tildeF}{\widetilde{\F}}
	\newcommand{\hatF}{\widehat{\F}}

	\newcommand{\R}{\mathbb{R}}
	
	\newcommand{\sspcoef}{\mathcal{C}}
	\newcommand{\tildesspcoef}{\widetilde{\mathcal{C}}}
	\newcommand{\hatsspcoef}{\widehat{\mathcal{C}}}

	\newcommand{\Dx}{\Delta x}
	\newcommand{\Dt}{\Delta t}
	\newcommand{\DtFE}{\Delta t_{\textnormal{FE}}}
	\newcommand{\Dtmax}{\Delta t_{\textnormal{max}}}
	\newcommand{\hatDtFE}{\widehat{\Delta t}_{\textnormal{FE}}}
	\newcommand{\tildeDtFE}{\widetilde{\Delta t}_{\textnormal{FE}}}

	\newcommand{\bv}[3]{\bm{b}^#1_{#2}\mkern-4mu\left(\mbox{\large $#3$}\right)}
	\newcommand{\Det}[1]{\textsc{det}\left(#1,\bm{a}_{i_1}, \dots, \bm{a}_{i_m},
						\bv{+}{j_1}{\frac{1}{\raisebox{-0.4ex}{$\scriptstyle\sspcoef$}}}\!, \dots,
						\bv{+}{j_n}{\frac{1}{\raisebox{-0.4ex}{$\scriptstyle\sspcoef$}}}\!,
						\bv{-}{l_1}{\frac{1}{\raisebox{-0.6ex}{$\scriptstyle\tildesspcoef$}}}\!, \dots,
						\bv{-}{l_s}{\frac{1}{\raisebox{-0.6ex}{$\scriptstyle\tildesspcoef$}}}\right)}
	\newcommand{\inset}[2]{\in \{#1, \dots, #2\}}
	
	\ifpaper
		\newcommand{\class}[1]{#1\unskip s}
	\else
		\newcommand{\class}[1]{#1\texorpdfstring{\unskip s}{s}}
	\fi

	\DeclareSymbolFont{tipa}{T3}{cmr}{m}{n}
	\DeclareMathAccent{\invbreve}{\mathalpha}{tipa}{16}

\begin{document}

\ifpaper
\else

	\title{Strong-stability-preserving additive linear multistep methods}
	\author{Yiannis Hadjimichael \and David I. Ketcheson\footnote{Author email addresses:
	\texttt{\{yiannis.hadjimichael, david.ketcheson\}@kaust.edu.sa}.
	This work was supported by the King Abdullah University of Science and Technology (KAUST), 4700
	Thuwal,23955-6900, Saudi Arabia.}}

	\maketitle

\fi

\begin{abstract}
	The analysis of strong-stability-preserving (SSP) linear multistep methods is extended to
	semi-discretized problems for which different terms on the right-hand side satisfy different forward
	Euler (or circle) conditions.
	Optimal additive and perturbed monotonicity-preserving linear multistep methods are studied in the
	context of such problems.
	Optimal perturbed methods attain larger monotonicity-preserving step sizes when the different forward
	Euler conditions are taken into account.
	On the other hand, we show that optimal SSP additive methods achieve a monotonicity-preserving
	step-size restriction no better than that of the corresponding non-additive SSP linear multistep
	methods.
\end{abstract}

\section{Introduction}\label{sec:Introduction}
We are interested in numerical solutions of initial value ODEs 
\begin{equation}\label{eq:ODE}
	\begin{split}
		\u'(t) &= \F(\u(t)), \quad t \ge t_0 \\
		\u(t_0) &= \u_0,
	\end{split}
\end{equation}
where $\F : \R^m \rightarrow \R^m$ is a continuous function and 
$\u :[t_0,\infty) \rightarrow \R^m$ satisfies a monotonicity property
\begin{align}\label{eq:MonotonicityExact}
	\|\u(t + \Dt)\| \le \|\u(t)\|, \quad \forall \Dt \ge 0,
\end{align}
with respect to some norm, semi-norm or convex functional $\|\cdot\| : \R^m \rightarrow \R$.
In general $\F(\u(t))$ may arise from the spatial discretization of partial differential equations; for
example, hyperbolic conservation laws.
A sufficient condition for monotonicity is that there exists some $\DtFE > 0$
such that the forward Euler condition
\begin{align}\label{eq:FECond}
	\|\u + \Dt\F(\u)\| \leq \|\u\|, \quad  0 \leq \Dt \leq \DtFE,
\end{align}
holds for all $\u \in \R^m$.

In this paper we focus on \class{\lmm} (LMMs) for the numerical integration of \eqref{eq:ODE}.
We denote by $\u_n$ the numerical approximation to $\u(t_n)$, evaluated sequentially at times
$t_n = t_0 + n\Dt$, $n \ge 1$.
At step $n$, a $k$-step \lmm applied to \eqref{eq:ODE} takes the form
\begin{align}\label{eq:LMM}
	\u_n = \sum_{j=0}^{k-1}\alpha_j\u_{n-k+j} + \Dt\sum_{j=0}^k\beta_j\F(\u_{n-k+j})
\end{align}
and if $\beta_k = 0$, then the method is explicit.

We would like to establish a discrete analogue of \eqref{eq:MonotonicityExact} for the numerical solution
$\u_n$ in \eqref{eq:LMM}.
Assuming $\F$ satisfies the forward Euler condition \eqref{eq:FECond} and all
$\alpha_j, \beta_j$ are non-negative, then convexity of $\|\cdot\|$ and the
consistency requirement $\sum_{j=0}^{k-1}\alpha_j =  1$ imply that
$\|\u_n\|\le \max_j\|\u_{n-k+j}\|$ whenever
$\Dt\beta_j/\alpha_j \le \DtFE$ for all $j$.
Hence, the monotonicity condition 
\begin{align}\label{eq:Monotonicity}
	\|\u_n\| \le \max\{\|\u_{n-1}\|, \dots, \|\u_{n-k}\|\}.
\end{align}
is satisfied under a step-size restriction
\begin{align}\label{eq:ts-restriction}
	\Dt \le \sspcoef_{\text{LMM}} \DtFE,
\end{align}
where $\sspcoef_{\text{LMM}} = \min_j \alpha_j/\beta_j$.
The ratio $\alpha_j/\beta_j$ is taken to be infinity if $\beta_j = 0$.
See \cite[Chapter~8]{Gottlieb:Ketcheson:Shu:2011:SSPbook} and
references therein for a review of strong-stability-preserving \class{\lmm} (SSP LMMs).

Most LMMs have one or more negative coefficients, so the foregoing analysis
leads to $\sspcoef_\text{LMM}=0$ and thus monotonicity condition \eqref{eq:Monotonicity} cannot be
guaranteed by positive step sizes.
However, typical numerical methods for hyperbolic conservation laws
$\U_t + \nabla\!\cdot\!\f(\U) = 0$ involve
upwind-biased semi-discretizations of the spatial derivatives.
In order to preserve monotonicity using methods with negative coefficients for
such semi-discretizations, downwind-biased spatial approximations may be used.
Let $\F$ and $\tildeF$ be respectively upwind- and downwind-biased approximations of
$-\nabla\!\cdot\!\f(\U)$.
It is natural to assume that $\tildeF$ satisfies
\begin{align}\label{eq:DFECond}
	\|\u - \Dt\F(\u)\| \leq \|\u\|, \quad  0 \leq \Dt \leq \DtFE,
\end{align}
for all $\u \in \R^m$.
A linear multistep method that uses both $\F$ and $\tildeF$ can be then written as
\begin{align}\label{eq:DLMM}
	\u_n = \sum_{j=0}^{k-1}\alpha_j\u_{n-k+j} +
        \Dt\sum_{j=0}^{k}\left(\beta_j\F(\u_{n-k+j}) - \tilde{\beta}_j\tildeF(\u_{n-k+j})\right).
\end{align}
If all $\alpha_j$ are non-negative, then the method is monotonicity preserving
under the restriction \eqref{eq:ts-restriction} where the SSP coefficient is
now $\tildesspcoef_{\text{LMM}} = \min_j \alpha_j/(\beta_j + \tilde{\beta}_j)$ with $\beta_j$,
$\tilde{\beta}_j$ non-negative; see \cite[Chapter~10]{Gottlieb:Ketcheson:Shu:2011:SSPbook} and the
references therein.

Downwind LMMs were originally introduced in
\cite{Shu:1988:TVD,Shu/Osher:1988:ENO}, with the idea that
$\F$ be replaced by $\tildeF$ whenever $\beta_j < 0$.
Optimal explicit linear multistep schemes of order up to six, coupled with efficient upwind and downwind
WENO discretizations, were studied in \cite{Gottlieb/Ruuth:2006:SSPDownwind}.
Coefficients of optimal upwind- and downwind-biased methods together with a reformulation of the nonlinear
optimization problem involved as a series of linear programming feasibility problems can be found in
\cite{Ketcheson:2009:OptimalMonotonicityGLM}.
Bounds on the maximum SSP step size for downwind-biased methods have been analyzed in
\cite{Ketcheson:2011:StepSizesDownwind}.

Method \eqref{eq:DLMM} can also be written in the perturbed form
\begin{align}\label{eq:PLMM}
    \u_n = \sum_{j=0}^{k-1}\alpha_j\u_{n-k+j} +
        \Dt\sum_{j=0}^k \left(\invbreve{\beta}_j\F(\u_{n-k+j}) +
        \tilde{\beta}_j\left(\F(\u_{n-k+j}) - \tildeF(\u_{n-k+j})\right)\right),
\end{align}
where $\invbreve{\beta}_j = \beta_j - \tilde{\beta}_j$.  We say method
\eqref{eq:PLMM} is a perturbation of the LMM \eqref{eq:LMM} with coefficients $\invbreve{\beta}_j$, and
the latter is referred to as the {\em underlying method} for \eqref{eq:PLMM}.
By replacing  $\tildeF$ with $\F$ in \eqref{eq:PLMM} one recovers the underlying method.
The notion of a perturbed method can be useful beyond the realm of downwinding for hyperbolic
PDE semi-discretizations.  If $\F$ satisfies the forward Euler condition
\eqref{eq:FECond} for both positive and negative step sizes, then we
can simply take $\tildeF = \F$.
In such cases, the perturbed and underlying methods are the same, but
analysis of a perturbed form of the
method can yield a larger step size for monotonicity, giving more
accurate insight into the behavior of the method.  See \cite{Higueras2010} for a
discussion of this in the context of Runge--Kutta methods, and see Example
\ref{ex:ExODE} herein for an example using multistep methods.
As we will see in Section~\ref{sec:DLMM}, the most useful \class{\plmm} \eqref{eq:PLMM} take a
form in which either $\beta_j$ or $\tilde{\beta}_j$ is equal to zero for each value of $j$.
Thus $\tildesspcoef_{\text{LMM}} = \min_j \{\alpha_j/\beta_j,\alpha_j/\tilde{\beta}_j\}$, and the
class of \class{\plmm} \eqref{eq:PLMM} coincides with the class of downwind LMMs in
\cite{Shu:1988:TVD,Shu/Osher:1988:ENO}.

In this work, we adopt form \eqref{eq:DLMM} for \class{\plmm} and consider
their application to the more general class of problems \eqref{eq:ODE} for
which $\F$ and $\tildeF$ satisfy forward Euler conditions under different
step-size restrictions:
\begin{subequations}\label{eq:DLMMDiffFE}
\begin{align}
	\|\u + \Dt\F(\u)\| \leq \|\u\|, &\quad \forall \u \in \R^m, \; 0 \leq \Dt \leq \DtFE
	\label{eq:DLMMDiffFEa} \\
	\|\u - \Dt\tildeF(\u)\| \leq \|\u\|, &\quad \forall \u \in \R^m, \; 0 \leq \Dt \leq \tildeDtFE.
	\label{eq:DLMMDiffFEb}
\end{align}
\end{subequations}
For a fixed order of accuracy and number of steps, an optimal SSP method is defined to be any method that
attains the largest possible SSP coefficient.
The choice of optimal monotonicity-preserving method for a given problem will
depend on the ratio $y = \DtFE/\tildeDtFE$.
We analyze and construct such optimal methods.
We illustrate by examples that \class{\plmm} with larger step sizes for monotonicity
can be obtained when the different step sizes in \eqref{eq:DLMMDiffFE} are
accounted for.

The perturbed methods \eqref{eq:DLMM} are reminiscent of additive methods, and
the latter can be analyzed in a similar way.  Consider the problem
\begin{align*}
    \u'(t) = \F(\u(t)) + \hatF(\u(t))
\end{align*}
where $\F$ and $\hatF$ may represent different physical processes, such
as convection and diffusion or convection and reaction.
Additive methods are expressed as
\begin{align*}
	\u_n = \sum_{j=0}^{k-1}\alpha_j\u_{n-k+j} +
	\Dt\sum_{j=0}^k\left(\beta_j\F(\u_{n-k+j}) + \hat{\beta}_j\hatF(\u_{n-k+j})\right),
\end{align*}
where $\F$ and $\hatF$ may satisfy the forward Euler condition \eqref{eq:FECond} under possibly different
step-size restrictions.
We prove that optimal SSP explicit or implicit additive methods have coefficients
$\beta_j = \hat{\beta}_j$ for all $j$, hence they lie within the class of
ordinary (not additive) LMMs.

The rest of the paper is organized as follows.
In Section~\ref{sec:DLMM} we analyze the monotonicity properties of \class{\plmm} for which
the upwind and downwind operators satisfy different forward Euler conditions.
Optimal methods are derived, and their properties are discussed.
Their effectiveness is illustrated by some examples.
Additive \class{\lmm} are presented in Section~\ref{sec:ALMM} where we prove that optimal SSP
\class{\almm} are equivalent to the corresponding non-additive SSP LMMs.
Monotonicity of IMEX \class{\lmm} is discussed, and finally in Section~\ref{sec:Conclusion} we summarize
the main results.

\section{Monotonicity-preserving perturbed \class{\lmm}} \label{sec:DLMM}
The following example shows that using upwind- and downwind-biased operators allows the construction
of methods that have positive SSP coefficients, even though the underlying methods are not SSP.
\begin{example}\label{ex:DLMM}
	Let $u'(t) = F(u(t))$ be a semi-discretization of $u_t + f(u)_x = 0$, where $F \approx -f(u)_x$.
	Consider the two-step, second-order explicit linear multistep method
	\begin{align}\label{eq:LMM32}
		u_n = \frac{1}{2}u_{n-2} - \frac{1}{4}\Dt F(u_{n-2}) + \frac{1}{2}u_{n-1} +
			\frac{7}{4}\Dt F(u_{n-1}).
	\end{align}
	The method has SSP coefficient equal to zero.
	Let us introduce a downwind-biased operator $\widetilde{F} \approx -f(u)_x$ such that
	\eqref{eq:DFECond} is satisfied.
	Then, a perturbed representation of \eqref{eq:LMM32} is
	\ifpaper
\else
	\begin{align}\label{eq:PLMM32}
		u_n = \frac{1}{2}u_{n-2} + \frac{1}{4}\Dt F(u_{n-2}) - \frac{1}{2}\Dt\widetilde{F}(u_{n-2}) +
			\frac{1}{2}u_{n-1} + 2\Dt F(u_{n-1}) - \frac{1}{4}\Dt\widetilde{F}(u_{n-1}),
	\end{align}
	\fi
	in the sense that the underlying method \eqref{eq:LMM32} is retrieved from \eqref{eq:PLMM32} by
	replacing $\widetilde{F}$ with $F$.
	The perturbed method has SSP coefficient $\tildesspcoef_{\text{LMM}} = 2/9$.
	There are infinitely many perturbed representations of \eqref{eq:LMM32}, but an optimal one is
	obtained by simply replacing $F$ with $\widetilde{F}$ in \eqref{eq:LMM32}, yielding
	\begin{align}\label{eq:DLMM32}
		u_n = \frac{1}{2}u_{n-2} - \frac{1}{4}\Dt\widetilde{F}(u_{n-2}) + \frac{1}{2}u_{n-1} +
					\frac{7}{4}\Dt F(u_{n-1}),
	\end{align}
	with SSP coefficient $\tildesspcoef_{\text{LMM}} = 2/7$.
\end{example}
\begin{remark}
        A LMM \eqref{eq:LMM} has SSP coefficient $\sspcoef=0$ if any of the
        following three conditions hold:
        \begin{enumerate}
            \item $\alpha_j<0$ for some $j$;
            \item $\beta_j<0$ for some $j$;
            \item $\alpha_j=0$ for some $j$ for which $\beta_j\ne 0$.
        \end{enumerate}
        By introducing a downwind operator we can remedy the second condition, but
        not the first or the third.  Most common methods, including the Adams--Bashforth,
        Adams--Moulton, and BDF methods, satisfy condition 1 or 3, so they cannot be
        made SSP via downwinding.
\end{remark}

We consider a generalization of the \class{\plmm} described previously, by assuming different forward
Euler conditions for the operators $\F$ and $\tildeF$ (see \eqref{eq:DLMMDiffFE}).
\begin{definition}\label{def:DLMMMonotonicity}
	A \plmm of the form \eqref{eq:DLMM} is said to be \textit{strong-stability-preserving}
	(SSP) with SSP coefficients $(\sspcoef,\tildesspcoef)$ if conditions
	\begin{align}\label{eq:DLMMMonCondDiffFE}
	\begin{split}
		\beta_j, \tilde{\beta}_j \geq 0, &\quad j \inset{0}{k}, \\
		\alpha_j - r\beta_j - \tilde{r}\tilde{\beta}_j \geq 0, &\quad j \inset{0}{k-1},
	\end{split}
	\end{align}
	hold for all $ 0 \le r \le \sspcoef$ and $0 \le \tilde{r} \le \tildesspcoef$.
\end{definition}
\noindent By plugging the exact solution in \eqref{eq:DLMM}, setting $\tildeF(\u(t_n)) = \F(\u(t_n))$ and
taking Taylor expansions around $t_{n-k}$, it can be shown that a \plmm is order $p$ accurate if
\begin{align}\label{eq:DLMMOrderCond}
\begin{gathered}
	\sum_{j=0}^{k-1}\alpha_j = 1, \qquad \sum_{j=0}^{k-1} j\alpha_j +
	\sum_{j=0}^k(\beta_j -\tilde{\beta}_j) = k, \\
	\sum_{j=0}^{k-1}\alpha_j j^i + \sum_{j=0}^k(\beta_j -\tilde{\beta}_j) i j^{i-1} = k^i, \quad
	i \inset{2}{p}.
\end{gathered}
\end{align}

The step-size restriction for monotonicity of an SSP \plmm is given by the following theorem.
\begin{theorem}\label{thm:DLMMMonotonicity}
	Consider an initial value problem for which $\F$ and $\tildeF$ satisfy the forward Euler
	conditions \eqref{eq:DLMMDiffFE} for some  $\DtFE > 0$, $\tildeDtFE > 0$.
	Let a consistent \plmm \eqref{eq:DLMM} be SSP with SSP coefficients $(\sspcoef, \tildesspcoef)$.
	Then the numerical solution satisfies the monotonicity condition \eqref{eq:Monotonicity}
	under a step-size restriction
	\begin{align}\label{eq:DLMMStepSize}
		\Dt \leq \min\{\sspcoef \,\DtFE, \tildesspcoef \, \tildeDtFE\}.
	\end{align}
\end{theorem}
\begin{proof}
	Define $\alpha_k  = \sspcoef\beta_k + \tildesspcoef\tilde{\beta_k}$ and add $\alpha_k \u_n$ to both
	sides of \eqref{eq:DLMM} to obtain
	\begin{align*}
            (1+\alpha_k)\u_n = \sum_{j=0}^{k}\left(\alpha_j \u_{n-k+j} - \Dt\beta_j\F(\u_{n-k+j}) +
		\Dt\tilde{\beta}_j\tildeF(\u_{n-k+j})\right).
	\end{align*}
	Since the method is SSP with coefficients $(\sspcoef,\tildesspcoef)$ then conditions
	\eqref{eq:DLMMMonCondDiffFE} hold for $r = \sspcoef$, $\tilde{r} = \tildesspcoef$.
	Let $\alpha_j = \hat{\alpha}_j + \tilde{\alpha}_j$
	with $\hat{\alpha}_j = \sspcoef \beta_j$.
	Then \eqref{eq:DLMMMonCondDiffFE} yields	$\tilde{\alpha}_j \ge \tildesspcoef\tilde{\beta}_j$
	and $\beta_j \ge 0, \tilde{\beta}_j \ge 0$.
	Thus, the right-hand side can be expressed as a convex combination of forward Euler steps:
	\ifpaper
\else
	\begin{align*}
            (1 + \alpha_k)\u_n = \sum_{j=0}^{k}\hat{\alpha}_j\Bigl(\u_{n-k+j} +
            \Dt\frac{\beta_j}{\hat{\alpha}_j}\F(\u_{n-k+j})\Bigr) +
                \sum_{j=0}^{k}\tilde{\alpha}_j\Bigl(\u_{n-k+j} -
                \Dt\frac{\tilde{\beta}_j}{\tilde{\alpha}_j}\tildeF(\u_{n-k+j})\Bigr).
	\end{align*}
	\fi
	Taking norms and using the triangle inequality yields
	\ifpaper
\else
	\begin{align*}
		(1 + \alpha_k)\|\u_n\| \leq \sum_{j=0}^{k}\hat{\alpha}_j\Bigl\|\u_{n-k+j} +
		\Dt\frac{\beta_j}{\hat{\alpha}_j}\F(\u_{n-k+j})\Bigr\| +
		\sum_{j=0}^{k}\tilde{\alpha}_j\Bigl\|\u_{n-k+j} -
		\Dt\frac{\tilde{\beta}_j}{\tilde{\alpha}_j}\tildeF(\u_{n-k+j})\Bigr\|.
	\end{align*}
	\fi
	Under the step-size restriction $\Dt \leq \min\{\sspcoef \,\DtFE,\tildesspcoef \, \tildeDtFE\}$ we
	get
	\begin{align*}
		\Dt\frac{\beta_j}{\hat{\alpha}_j} \le \DtFE \quad  \text{and}
		\quad \Dt\frac{\tilde{\beta}_j}{\tilde{\alpha}_j} \le \tildeDtFE.
	\end{align*}
	Since $\F$ and $\tildeF$ satisfy \eqref{eq:DLMMDiffFEa} and
	\eqref{eq:DLMMDiffFEb} respectively, we have
	\begin{align*}
		(1 + \alpha_k)\|\u_n\| &\leq \sum_{j=0}^{k}\hat{\alpha}_j\|\u_{n-k+j}\| +
		\sum_{j=0}^{k}\tilde{\alpha}_j\|\u_{n-k+j}\|,
	\end{align*}
	and hence
	\begin{align*}
		\|\u_n\| \le \sum_{j=0}^{k-1}\alpha_j\|\u_{n-k+j}\| \le
		\max_{0 \leq j \leq k-1}\|\u_{n-k+j}\| \sum_{j=0}^{k-1}\alpha_j.
	\end{align*}
        Consistency requires $\sum_{j=0}^{k-1}\alpha_j = 1$ and therefore the
        monotonicity condition \eqref{eq:Monotonicity} follows.
\end{proof}

\subsection{Optimal SSP perturbed \class{\lmm}}
We now turn to the problem of finding, among methods with a given number of steps $k$
and order of accuracy $p$, the largest SSP coefficients.  Since $\sspcoef$, $\tildesspcoef$ 
are continuous functions of the method's coefficients, we expect that the maximal step size
\eqref{eq:DLMMStepSize} is achieved when $\sspcoef = \tildesspcoef \, \tildeDtFE/\DtFE$.
It is thus convenient to define $y := \DtFE/\tildeDtFE$.
\begin{definition}\label{def:DLMMSSPCoeff}
For a fixed $y \in [0,\infty)$ we say that an SSP method \eqref{eq:DLMM} has \textit{SSP coefficient}
\begin{align*}
	\sspcoef(y) = \sup \bigl\{ r \ge 0 : \text{monotonicity conditions 
	\eqref{eq:DLMMMonCondDiffFE} hold with } \tilde{r} = y r \bigr\}
\end{align*}
and its corresponding downwind SSP coefficient is $\tildesspcoef(y) = y \, \sspcoef(y)$.
Given a number of steps $k$ and order of accuracy $p$ an SSP method is called \textit{optimal}, if it has
SSP coefficient
\ifpaper
\else
	\begin{align*}
		\sspcoef_{k,p}(y) = \sup_{\bm{\alpha},\bm{\beta},\bm{\tilde{\beta}}} \bigl\{ \sspcoef(y) > 0 :
		\sspcoef(y) \text{ is the SSP coefficient of a $k$-step method \eqref{eq:DLMM} of order $p$}
		\bigr\}.
	\end{align*}
\fi
\end{definition}
Next we prove that for a given SSP \plmm with SSP coefficient $\sspcoef(y)$, we can construct another SSP
method \eqref{eq:DLMM} with the property that for each $j$, either $\beta_j$ or $\tilde{\beta}_j$
is zero.
Example~\ref{ex:DLMM} is an application of this result.
\begin{lemma}\label{lem:BetaTimesBeta}
	Consider a $k$-step \plmm \eqref{eq:DLMM} of order $p$ with SSP coefficient $\sspcoef(y)$ for a
	given $y$.
	Then, we can construct a $k$-step SSP method \eqref{eq:DLMM} of order $p$ with SSP coefficient at
	least $\sspcoef(y)$ that satisfies $\beta_j\tilde{\beta}_j = 0$ for each $j$.
	Moreover, both perturbed methods correspond to the same underlying method.
\end{lemma}
\begin{proof}
	Suppose there exists an $k$-step SSP method \eqref{eq:DLMM} of order $p$ with SSP coefficient
	$\sspcoef(y)$ for some $y \in [0,\infty)$, such that $\beta_j \ge \tilde{\beta}_j > 0$ for
	$j \in J_1 \subseteq \{0,1,\dots,k\}$ and $\tilde{\beta}_j > \beta_j > 0$ for
	$j \in J_2 \subseteq \{0,1,\dots,k\}$.
	Clearly $J_1 \cap J_2 = \emptyset$.
	Define 
	\begin{align*}
		\beta_j^* = \begin{cases}
								\beta_j - \tilde{\beta}_j, &\mbox{if } j \in J_1, \\
								0, &\mbox{if } j \notin J_1,
							\end{cases}
		\qquad
		\tilde{\beta}_j^*  = \begin{cases}
										0, &\mbox{if } j \notin J_2, \\
										\tilde{\beta}_j - \beta_j, &\mbox{if } j \in J_2.
									\end{cases}		
	\end{align*}
	Observe that conditions \eqref{eq:DLMMMonCondDiffFE} with
	$r = \sspcoef(y)$, $\tilde{r} = \tildesspcoef(y)$ and the order conditions
	\eqref{eq:DLMMOrderCond} are satisfied when $\beta_j,\tilde{\beta}_j$ are replaced by
	$\beta_j^*,\tilde{\beta}_j^*$.
	Therefore, the method with coefficients $(\bm{\alpha},\bm{\beta^*} ,\bm{\tilde{\beta^*}})$ has SSP
	coefficient at least $\sspcoef(y)$ and satisfies $\beta_j^*\tilde{\beta}_j^* = 0$ for each $j$.
	Finally, the definition of $\beta_j^*$ and $\tilde{\beta}_j^*$ leaves $\beta_j - \tilde{\beta}_j$
	invariant, thus substituting $\tildeF = \F$ in method \eqref{eq:DLMM} with coefficients
	$(\bm{\alpha},\bm{\beta} ,\bm{\tilde{\beta}})$ or
	$(\bm{\alpha},\bm{\beta^*} ,\bm{\tilde{\beta^*}})$ yields the same underlying method.
\end{proof}

The next Corollary is an immediate consequence of Lemma~\ref{lem:BetaTimesBeta}.
\begin{corollary}\label{col:OptimalBetaTimesBeta}
	Let $k$, $p$ and $y$ be given such that $\sspcoef_{k,p}(y) > 0$.
	Then there exists an optimal SSP \plmm \eqref{eq:DLMM} with SSP coefficient $\sspcoef_{k,p}(y)$
	that satisfies $\beta_j\tilde{\beta}_j = 0$ for each $j$.
\end{corollary}
Based on Lemma~\ref{col:OptimalBetaTimesBeta} we have the following upper bound for the SSP coefficient of
any \plmm \eqref{eq:DLMM}.
This extends Theorem~2.2 in \cite{Ketcheson:2011:StepSizesDownwind}.
\begin{theorem}\label{thm:SSPCoefUpperBound}
	Given $y \in [0,\infty)$, any \plmm \eqref{eq:DLMM} of order greater than one satisfies
	$\sspcoef(y) \le 2$.
\end{theorem}
\begin{proof}
	Consider a second-order optimal SSP \plmm with SSP coefficient $\sspcoef = \sspcoef(y)$ and
	$\tildesspcoef = y \, \sspcoef(y)$ for some $y \in [0,\infty)$.
	Then, from Lemma~\ref{col:OptimalBetaTimesBeta} there exists an optimal method with the at least SSP
	coefficient $\sspcoef$ and coefficients $(\bm{\alpha},\bm{\beta},\tilde{\bm{\beta}})$ such that
	$\beta_j\tilde{\beta}_j = 0$ for each $j$.
	Suppose $y > 0$ and define $\delta_j = \beta_j + y\tilde{\beta}_j$ and
	\begin{align*}
		\sigma_j =
			\begin{cases} 
				 1 &\mbox{if } \tilde{\beta}_j = 0 \\
				 -1/y & \mbox{if } \beta_j = 0.
			\end{cases}
	\end{align*}
	Since either $\beta_j$ or $\tilde{\beta}_j$ is zero, then
	$\beta_j - \tilde{\beta}_j = \sigma_j\delta_j$ for all $j$.
	Let $\gamma_j = \alpha_j - \sspcoef\delta_j$ for $j \inset{0}{k-1}$.
	Taking $p=2$, $r = \sspcoef$, and $\tilde{r} = \tildesspcoef$  in \eqref{eq:DLMMOrderCond}, the
	second order conditions can be written as
	\begin{align}
		&\sum_{j=0}^{k-1} \gamma_j + \sspcoef\delta_j = 1, \label{eq:DLMMOrder2a} \\
		&\sum_{j=0}^{k-1} j\gamma_j + (j\sspcoef+\sigma_j)\delta_j = k - \sigma_k\delta_k,
		\label{eq:DLMMOrder2b} \\
		&\sum_{j=0}^{k-1} j^2\gamma_j + (j^2\sspcoef+2j\sigma_j)\delta_j = k(k - 2\sigma_k\delta_k).
		\label{eq:DLMMOrder2c}
	\end{align}
	Multiplying \eqref{eq:DLMMOrder2a}, \eqref{eq:DLMMOrder2b} and \eqref{eq:DLMMOrder2c}
	by $-k^2$, $2k$ and $-1$, respectively and adding all three expressions gives
	\begin{align}\label{eq:OrderCondExpression}
		\sum_{j=0}^{k-1} -(k-j)^2\gamma_j + \left(-\sspcoef(k-j)^2 + 2\sigma_j(k-j)\right)\delta_j = 0.
	\end{align}
	Since the method satisfies conditions \eqref{eq:DLMMMonCondDiffFE} for $r = \sspcoef$ and
	$\tilde{r} = \tildesspcoef$, then all coefficients $\gamma_j$ and $\delta_j$ are non-negative.
	Therefore, there must be at least one index $j_0$ such that the coefficient of $\delta_{j_0}$ in
	\eqref{eq:OrderCondExpression} is non-negative.
	Note that if $\beta_{j_0} = 0$, then $\sigma_{j_0} < 0$; hence it can only be that
	$\tilde{\beta}_{j_0} = 0$ and $\beta_{j_0} \neq 0$.
	Thus,
	\begin{align*}
		-\sspcoef(k-j_0)^2 + 2(k-j_0) \ge 0,
	\end{align*}
	which implies 
	\begin{align}\label{eq:SSPUpperBound}
		\sspcoef \le \frac{2}{k-j_0} \le 2
	\end{align}
	since $k - j_0 \ge 1$.
	If now $y = 0$, define $\delta_j = \beta_j + \tilde{\beta}_j$ and
	$\sigma_j = \text{sign}(\beta_j - \tilde{\beta}_j)$.
	Using $\gamma_j = \alpha_j - \sspcoef\beta_j$ and performing the same algebraic manipulations as
	before we get
	\begin{align}\label{eq:OrderCondExpression2}
		\sum_{j=0}^{k-1} -(k-j)^2(\gamma_j + \sspcoef\beta_j) + 2\sigma_j(k-j)\delta_j = 0.
	\end{align}
	Again, there must be at least one index $j_0$ in \eqref{eq:OrderCondExpression2} for which the
	coefficient of $\delta_{j_0}$ is non-negative, thus $\delta_{j_0} = \beta_{j_0} \neq 0$ and this
	yields the inequality \eqref{eq:SSPUpperBound}.
\end{proof}

\begin{remark}
	For given values $k, p, y$, it may be that there exists no method with positive SSP coefficients.
	However, from \eqref{eq:DLMMMonCondDiffFE} and Theorem~\ref{thm:SSPCoefUpperBound} if a method exists
	with bounded SSP coefficient, then the existence of an optimal method follows since the feasible
	region is compact.
\end{remark}
By combining conditions \eqref{eq:DLMMMonCondDiffFE} and \eqref{eq:DLMMOrderCond}, and
setting
\begin{align}\label{eq:DLMMGamma}
	\gamma_j = \alpha_j - r\beta_j - \tilde{r}\tilde{\beta}_j \quad \text{for} \quad j \inset{0}{k-1},
\end{align}
the problem of finding optimal SSP \class{\plmm} \eqref{eq:DLMM} can be formulated as a linear
programming feasibility problem:
\begin{lp}\label{DLMMFeasibilityProblem}
	For fixed $k \ge 1$, $p \ge 1$ and a given $y \in [0,\infty)$, determine whether there exist
	non-negative coefficients $\gamma_j$, $j \in \{0, \dots, k-1\}$ and $\beta_j, \tilde{\beta}_j$,
	$j \in \{0, \dots, k\}$ such that
	\begin{align}\label{eq:DLMMFeasibilityCond}
	\begin{gathered}
		\sum_{j=0}^{k-1} \gamma_j + r\beta_j + \tilde{r}\tilde{\beta}_j = 1, \qquad
		\sum_{j=0}^{k-1} j(\gamma_j + r\beta_j + \tilde{r}\tilde{\beta}_j) +
		\sum_{j=0}^k(\beta_j -\tilde{\beta}_j) = k, \\
		\sum_{j=0}^{k-1}(\gamma_j + r\beta_j + \tilde{r}\tilde{\beta}_j) j^i +
		\sum_{j=0}^{k}(\beta_j - \tilde{\beta}_j) ij^{i-1} = k^i, \quad i \inset{2}{p},
	\end{gathered}
	\end{align}
	for some value $ r \ge 0$ and $\tilde{r} = yr$.
\end{lp}
Expressing \eqref{eq:DLMMFeasibilityCond} in a compact form facilitates the analysis of the feasible
problem LP~\ref{DLMMFeasibilityProblem}.
Let the vector
\begin{align}\label{eq:DLMMVectora}
	\bm{a}_j := (1,j,j^2,\dots,j^p)^\intercal \in \R^{p+1},
\end{align}
and denote by $\bm{a}'_j$ the derivative of $\bm{a}_j$ with respect to $j$, namely
$\bm{a}'_j = (0, 1, 2j, \allowbreak\dots, pj^{p-1})^\intercal$.
Define
\begin{align}\label{eq:DLMMVectorb}
	\bm{b}^\pm_j(x) := \begin{cases}
						\pm x\bm{a}'_k & \mbox{if } j = k, \\
						\bm{a}_j \pm x\bm{a}'_j &\mbox{otherwise.}
					\end{cases}
\end{align}
The conditions \eqref{eq:DLMMFeasibilityCond} can be expressed in terms of vectors
$\bm{a}_j,\bm{b}^\pm_j(\cdot)$:
\begin{align}\label{eq:DLMMOrderCondVectorForm}
	\sum_{j=0}^{k-1} \gamma_j\bm{a}_j +
	r \sum_{j=0}^k \beta_j \bm{b}^+_j(r^{-1}) +
	\tilde{r} \sum_{j=0}^k \tilde{\beta}_j\bm{b}^-_j(\tilde{r}^{-1}) = \bm{a}_k.
\end{align}
The number of non-zero coefficients of an optimal SSP \plmm is given by
Theorem~\ref{thm:DLMMNonzeroCoeff}.
The following lemma is a consequence of Carath\'{e}odory's theorem, which states that if a vector $\bm{x}$
belongs to the convex hull of a set $S \subseteq \R^n$, then it can be expressed as a convex combination
of $n+1$ vectors in $S$.
The proof appears in Appendix~\ref{appx:ProofLemmas}.
\begin{lemma}\label{lem:DLMMConvexity}
	Consider a set $S = \{\bm{x}_1, \dots, \bm{x}_m\}$ of distinct vectors $\bm{x}_j \in \R^n$,
	$j \inset{1}{m}$.
	Let $C = \text{conv}(S)$ be the convex hull of $S$.
	Then the following statements hold:
	\begin{enumerate}[label=(\alph*)]
		\item Any non-zero vector in $C$ can be expressed as a non-negative linear combination of at most
			$n$ linearly independent vectors in $S$.
		\item Suppose the vectors in $S$ lie in the hyperplane $\{(1,\bm{v}) : \bm{v} \in \R^{n-1}\}$ of 
			$\R^n$.
			Then any non-zero vector in $C$ can be expressed as a convex combination of at most $n$
			linearly independent vectors in $S$.
	\end{enumerate}
\end{lemma}
\begin{theorem}\label{thm:DLMMNonzeroCoeff}
	Let $k, p$ be positive integers such that $0 < \sspcoef_{k,p}(y) < \infty$ for a given
	$y \in [0,\infty)$.
	Then there exists an optimal \plmm \eqref{eq:DLMM} with SSP coefficient
	$\sspcoef = \sspcoef_{k,p}(y)$ that has at most $p$ non-zero coefficients $\gamma_i$,
	$i \inset{0}{k-1}$ and $\beta_j, \tilde{\beta}_j$, $j \inset{0}{k}$.
\end{theorem}
\begin{proof}
	Consider an optimal LMM \eqref{eq:DLMM} with coefficients
	$(\bm{\alpha},\bm{\beta},\bm{\tilde{\beta}})$ and SSP coefficient $\sspcoef_{k,p}(y) > 0$,
	for a given $y \in [0,\infty)$.
	From Lemma~\ref{col:OptimalBetaTimesBeta} an optimal method can be chosen such that
	$\beta_j\tilde{\beta}_j = 0$ for each $j$.
	Using \eqref{eq:DLMMGamma} we can perform a change of variables and consider the vector of
	coefficients
	$\bm{x}(r) = \bigl(\bm{\gamma}(r), \bm{\beta}(r), \bm{\tilde{\beta}}(r)\bigr)^\intercal \in
	\R^{3k+2}$, $\bm{x}(r) \ge 0$.
	We will show that $\bm{x}$ has at most $p$ non-zero coefficients.
	Suppose on the contrary that $\bm{x}$ has at least $p+1$ non-zero coefficients
	\begin{align*}
		\gamma_{i_1}, \dots,\gamma_{i_m}, \beta_{j_1}, \dots, \beta_{j_n}, \tilde{\beta}_{l_1}, \dots,
		\tilde{\beta}_{l_s},
	\end{align*}
	where $0 \le i_1 < \dots < i_m \le k-1$, $0 \le j_1 < \dots < j_n \le k$ and
	$0 \le l_1 < \dots < l_s \le k$.

	Assume that the set
	\begin{align*}
		S = \left\{\bm{a}_{i_1}, \dots, \bm{a}_{i_m}, \bv{+}{j_1}{\frac{1}{r}}, \dots,
		\bv{+}{j_n}{\frac{1}{r}},\bv{-}{l_1}{\frac{1}{yr}}, \dots, \bv{-}{l_s}{\frac{1}{yr}}\right\}
	\end{align*}
	spans $\R^{p+1}$.
	Let $\tilde{r} = yr$; then the system of equations \eqref{eq:DLMMOrderCondVectorForm} can
	be written as $A(r)\bm{x}(r) = \bm{a}_k$, where
	\begin{align*}
		A(r) =
		\def\arraystretch{2.5}
		\left[\begin{array}{c|c|c|c|c|c|c|c|c}
			\bm{a}_{i_1} & \dots & \bm{a}_{i_m} & r\bv{+}{j_1}{\frac{1}{r}} & \dots &
			r\bv{+}{j_n}{\frac{1}{r}} & yr\bv{-}{l_1}{\frac{1}{yr}} & \dots &
			yr\bv{-}{l_s}{\frac{1}{yr}} \\[1em]
		\end{array}\right].
	\end{align*}
	Let $\bm{x}_p = (\bm{x}_B, \bm{x}_N)^\intercal$ be a permutation of $\bm{x}$ such that
	$\bm{x}_B^\intercal \in \R^{p+1}$ is a strictly positive vector and
	$\bm{x}_N^\intercal \in \R^{3k-p+1}$ is non-negative.
	The columns of $A(r)$ can be permuted in the same way, yielding
	$A_p(r) = [B(r) \;|\; N(r)]$, where $B \in \R^{(p+1) \times (p+1)}$ and
	$N \in \R^{(p+1) \times (3k-p+1)}$.
	Hence, the columns of $B$ and $N$ are associated with  $\bm{x}_B$ and $\bm{x}_N$, respectively.
	From our assumption there must be a subset of $S$ that forms a basis for $\R^{p+1}$, hence $A(r)$ can
	be permuted in such a way so that $B(r)$ has full rank.
	Therefore, $A_p(r)\bm{x}_p(r) = \bm{a}_k$ gives
	$\bm{x}^\intercal_B(r) = B^{-1}(r)\bigr(\bm{a}_k - N(r)\bm{x}^\intercal_N\bigr)$.
	Since $\bm{x}_B(r)  > 0$, there exists $\epsilon > 0$ such that
	$\bm{x}^*_B = \bm{x}_B(r + \epsilon) > 0$.
	Note that we can choose to perturb only $\bm{x}_B$ and keep $\bm{x}_N$ invariant.
	Let $\bm{x}^*_p = (\bm{x}^*_B, \bm{x}_N)^\intercal$, then $A_p(r +\epsilon)\bm{x}^*_p = \bm{a}_k$.
	But this contradicts to the optimality of the method since we can construct a  $k$-step SSP \plmm of
	order $p$ and coefficients given by $\bm{x}^*$ and SSP coefficient $\sspcoef_{k,p}(y) + \epsilon$.

	Now, assume that the set $S$ does not span $\R^{p+1}$.
	Then the vectors in set $S$ lie in the hyperplane $\{(1,v) : v \in \R^p\} \subset \R^{p+1}$ and they
	are linearly dependent.
	If the method is explicit then $\beta_k = \tilde{\beta}_k = 0$ and $\bm{a}_k$ lies in the convex hull
	of $S$.
	Therefore, from part (b) of Lemma~\ref{lem:DLMMConvexity} the vector $\bm{a}_k$ can be expressed as a
	convex combination of $p$ vectors in $S$.
	In the case the method is implicit, assume without loss of generality that $\beta_k > 0$ and
	divide \eqref{eq:DLMMOrderCondVectorForm} by $(1+r\beta_k)$.
	The vector $(1+r\beta_k)^{-1}\bm{a}_k$ belongs to the convex hull of $S$ and thus from part (a) of
	Lemma~\ref{lem:DLMMConvexity} it can be written as a non-negative linear combination of $p$ vectors
	in $S$.
\end{proof}
Furthermore, uniqueness of optimal \class{\plmm} can be established under certain conditions
on the vectors $\bm{a}_j, \bm{b}^\pm_j$.
The following lemma is a generalization of \cite[Lemma~3.5]{Lenferink:1989:ExLMM}.
\begin{lemma}\label{lem:DLMMUniqueness}
	Consider an optimal \plmm \eqref{eq:DLMM} with SSP coefficient $\sspcoef = \sspcoef_{k,p}(y) > 0$ and
	$\tildesspcoef = y \, \sspcoef_{k,p}(y)$ for a given $y \in [0,\infty)$.
	Let the indices
	\begin{align*}
		0 \le i_1 < \dots < i_m \le k-1, \quad 0 \le j_1 < \dots < j_n \le k,
		\quad 0 \le l_1 < \dots < l_s \le k,
	\end{align*}
	where $m+n+s \le p$ be such that
	$\gamma_{i_1}, \dots, \gamma_{i_m}, \beta_{j_1} \dots, \beta_{j_n}, \tilde{\beta}_{l_1}, \dots,
	\tilde{\beta}_{l_s}$
	are the positive coefficients in \eqref{eq:DLMM}.
	Let us also denote the sets $I = \{0, \dots, k\}$, $I_1 = \{i_1, \dots, i_m\}$,
	$J_1 = \{j_1, \dots, j_n\}$, $J_2 = \{l_1, \dots, l_s\}$.
	Assume that the function
	\begin{align*}
		F(v) = \Det{\bm{v}}
	\end{align*}
	is either strictly positive or strictly negative, simultaneously for all
	$\bm{v} = \bm{a}_i$, $i \in I \setminus (I_1 \cup \{k\})$,
	$\bm{v} = \bv{+}{j}{{\mbox{\small $1/\sspcoef$}}}$, $j \in I \setminus J_1$ and
	$\bm{v} = \bv{-}{l}{{\mbox{\small $1/\tildesspcoef$}}}$, $l \in I \setminus J_2$.
	Then \eqref{eq:DLMM} is the unique optimal $k$-step SSP \plmm of order $p$.
\end{lemma}
\begin{proof}
	Assume there exists another optimal $k$-step method of order least $p$ with coefficients
	$(\bm{\alpha^*},\bm{\beta^*},\bm{\tilde{\beta}^*})$.
	Define $\gamma_i^* = \alpha_i^* - \sspcoef\beta_i^* - \tildesspcoef\tilde{\beta}_i^*$,
	$i \inset{0}{k-1}$, then by the monotonicity conditions \eqref{eq:DLMMMonCondDiffFE} and
	Definition~\ref{def:DLMMSSPCoeff} we have
	\begin{align*}
		\gamma_i^* \ge 0 &\qquad\qquad i \inset{0}{k-1}, \\
		\beta_j^* \ge 0, \tilde{\beta}_j^* \ge 0 &\qquad\qquad j \inset{0}{k}, \\
		\sum_{i=0}^{k-1} \gamma_i^*\bm{a}_i +
		\sspcoef \sum_{j=0}^k \beta_j^* &\bv{+}{j}{\frac{1}{\raisebox{-0.4ex}{$\scriptstyle\sspcoef$}}}
		+ \tildesspcoef \sum_{j=0}^k \tilde{\beta}_j^*
		\bv{-}{j}{\frac{1}{\raisebox{-0.6ex}{$\scriptstyle\tildesspcoef$}}} = \bm{a}_k.
	\end{align*}
	Since the method \eqref{eq:DLMM} with coefficients $(\bm{\alpha},\bm{\beta},\bm{\tilde{\beta}})$ is
	optimal, then $\bm{a}_k$ can be also written as a linear combination of vectors
	\begin{align}\label{eq:DLMMVectors}
		\bm{a}_{i_1}, \dots, \bm{a}_{i_m},
		\bv{+}{j_1}{\frac{1}{\raisebox{-0.4ex}{$\scriptstyle\sspcoef$}}} \dots,
		\bv{+}{j_n}{\frac{1}{\raisebox{-0.4ex}{$\scriptstyle\sspcoef$}}},
		\bv{-}{l_1}{\frac{1}{\raisebox{-0.6ex}{$\scriptstyle\tildesspcoef$}}}, \dots,
		\bv{-}{l_s}{\frac{1}{\raisebox{-0.6ex}{$\scriptstyle\tildesspcoef$}}},
	\end{align}
	and moreover from Lemma~\ref{lem:DLMMConvexity} the vectors in \eqref{eq:DLMMVectors} are
	linearly independent.
	Hence,
	\begin{align*}
		0 =& \Det{\bm{a}_k} \\
		  =& \sum_{i=0}^{k-1} \gamma_i^*\,\Det{\bm{a}_i} + \\
		  &\sspcoef \sum_{j=0}^k \beta_j^*\,
		  \Det{\bv{+}{j}{\frac{1}{\raisebox{-0.4ex}{$\scriptstyle\sspcoef$}}}} + \\
		  &\tildesspcoef \sum_{j=0}^k \tilde{\beta}_j^*\,
		  \Det{\bv{-}{j}{\frac{1}{\raisebox{-0.6ex}{$\scriptstyle\tildesspcoef$}}}}.
	\end{align*}
	By positivity of coefficients $\gamma_i^*, \beta_j^*, \tilde{\beta}_j^*$ and the assumptions of the
	lemma, we have $\gamma_i^* = 0$, $i \notin I_1$, $\beta_j^* = 0$, $j \notin J_1$ and
	$\tilde{\beta}_j^* = 0$, $j \notin J_2$.
	Linear independence of the vectors in \eqref{eq:DLMMVectors} implies that $\gamma_i^* = \gamma_i$,
	$i \in I_1$ and $\beta_j^* = \beta_j$,  $j \in J_1$ and $\tilde{\beta}_j^* = \tilde{\beta}_j$,
	$j \in J_2$ and the statement of the lemma is proved.
\end{proof}
Fixing the number of steps $k$, and the order of accuracy $p$, the feasibility problem
LP~\ref{DLMMFeasibilityProblem} has been numerically solved for different values of $y$, by using
\texttt{linprog} from \texttt{MATLAB}'s optimization toolbox.
Optimal explicit and implicit \class{\plmm} are found for $k \inset{1}{40}$ and $p \inset{1}{15}$.
\begin{remark}
	In all cases we have investigated, the SSP coefficient $\sspcoef(y)$
	(see Definition~\ref{def:DLMMSSPCoeff}) is a strictly decreasing function.
	Similarly, the corresponding SSP coefficient $\tildesspcoef(y)$ is strictly
	increasing.  This suggests that whenever $\F$ and
	$\tildeF$ satisfy \eqref{eq:DLMMDiffFE}, then for a fixed number
	of stages and order of accuracy, the optimal \plmm obtained by
	considering the different step sizes in \eqref{eq:DLMMDiffFE}, allows \emph{larger} step sizes for
	monotonicity than what is allowed by the optimal downwind SSP method obtained
	just by taking the minimum of the two forward Euler step sizes.
	This behavior is shown in Figure~\ref{fig:DLMM} for the class of
	two-step, second-order \class{\plmm}.
\end{remark}
\begin{remark}
	The dependence of the SSP coefficient $\sspcoef(y)$ with respect to $y$ can be explained in view of
	equations \eqref{eq:DLMMFeasibilityCond} and forward Euler conditions \eqref{eq:DLMMDiffFE}.
	As $y$ approaches zero, the step-size restriction in \eqref{eq:DLMMDiffFEa} becomes more
	severe, but \eqref{eq:DLMMFeasibilityCond} depends less on coefficients $\tilde{\beta}_j$ enabling
	larger SSP coefficients to be obtained.
	On the other hand, as $y$ tends to infinity the step-size restriction of forward Euler condition
	\eqref{eq:DLMMDiffFEb} is stricter and coefficients $\tilde{\beta}_j$ tend to zero.
	In other words, the best possible SSP method in this case would be a method without downwind and thus
	the SSP coefficient $\sspcoef(y)$ approaches the corresponding SSP coefficient of traditional
	LMMs \eqref{eq:LMM}.
\end{remark}
\subsection{Examples}
Here we illustrate the effectiveness of \class{\plmm} by presenting two
examples.  We consider the following assumptions:
\begin{enumerate}
	\item Condition \eqref{eq:FECond} holds only for operator $\F$;
	\item Conditions \eqref{eq:DLMMDiffFE} hold for $\F$ and $\tildeF$ under a step-size restriction
		$\Dt \le \min\{\DtFE,\tildeDtFE\}$;
	\item Conditions \eqref{eq:DLMMDiffFE} hold for $\F$ and $\tildeF$ under different step-size
		restrictions.
\end{enumerate}
In the literature, traditional SSP LMMs applied to problems satisfying assumption (1) have been
extensively studied, for example see \cite{Lenferink:1989:ExLMM, Lenferink:1991:ImLMM,
Hundsdorfer/Ruuth/Spiteri:2003:MonotonicityLMM}.
Downwind SSP LMMs \cite{Shu:1988:TVD,Shu/Osher:1988:ENO,
Ruuth/Hundsdorfer:2005:GeneralMonotonicityBoundednessLMM, Ketcheson:2009:OptimalMonotonicityGLM,
Ketcheson:2011:StepSizesDownwind} were introduced for problems that comply with assumption (2), whereas
methods for problems satisfying assumption (3) are the topic of this work.
\begin{example}\label{ex:ExODE}
	Consider the ODE problem
	\begin{equation}\label{eq:ExODE}
		\begin{split}
			u'(t) &= u(t)^2(u(t)-1), \quad t \ge 0 \\
			u(t_0) &= u_0.
		\end{split}
	\end{equation}
	The right-hand side is Lipschitz continuous in $u$ in a close interval containing $[0,1]$.
	Thus, there exists a unique solution and it is easy to see that existence holds for all $t$.
	Therefore, if $u(t_0) = 0$ or $u(t_0) = 1$, then $u(t) = 0$ or $u(t) = 1$, respectively for all $t$.
	If $u_0 \in [0,1]$, uniqueness implies that $u(t) \in [0,1]$ for all $t$.
	It can be also shown that if $u \in (0,1]$, then
	\begin{align*}
            0 & \le u + \Dt\, u^2(u-1) \le 1 \quad \text{for } 0 \le \Dt \le 4, \\
            0 & \le u - \Dt\, u^2(u-1) \le 1 \quad \text{for } 0 \le \Dt \le 1.
	\end{align*}
	Applying method \eqref{eq:DLMM} where $\F = u^2(u-1)$, it is natural to take $\tildeF = \F$, and
	then we have that \eqref{eq:DLMMDiffFE} holds with $\DtFE = 4$ and $\tildeDtFE = 1$.
	For method \eqref{eq:LMM32}, in practice we observe that $u_n \in 0,1]$
	whenever $\Dt \le 8/7$.  The method has $\sspcoef_\text{LMM}=0$,
	so applying only assumption (1) above we cannot expect a monotone
	solution under any step size.  Using assumption (2), and writing
	the method in the form \eqref{eq:DLMM32} (notice that perturbations do
	not change the method at all in this case, since $\tildeF = \F$)
	we obtain a step-size restriction
	$\Dt \le \tildesspcoef_{\text{LMM}}\min\{\DtFE,\tildeDtFE\} = 2/7$, since
	$\tildesspcoef_{\text{LMM}} = 2/7$.
	Finally, using assumption (3) to take into account the different
	forward Euler step sizes for $\F$ and $\tildeF$, we obtain the
	step-size restriction $\Dtmax = \tildesspcoef_{\text{LMM}}  \DtFE = 8/7$,
	which matches the experimental observation.

	An even larger step-size restriction can be achieved by finding the
	optimal \plmm among the class of two-step, second-order \class{\plmm}.
	In this case $y = \DtFE/\tildeDtFE = 4$ and the optimal \plmm has SSP coefficient
	$\sspcoef_{2,2}(4) = 0.3465$, thus the numerical solution is guaranteed to lie in the interval
	$[0,1]$ if the step size is at most $\Dtmax = \sspcoef_{2,2}(4) \DtFE = 1.386$.
\end{example}

\begin{figure}
	\centering
    \includegraphics[width=0.55\textwidth]{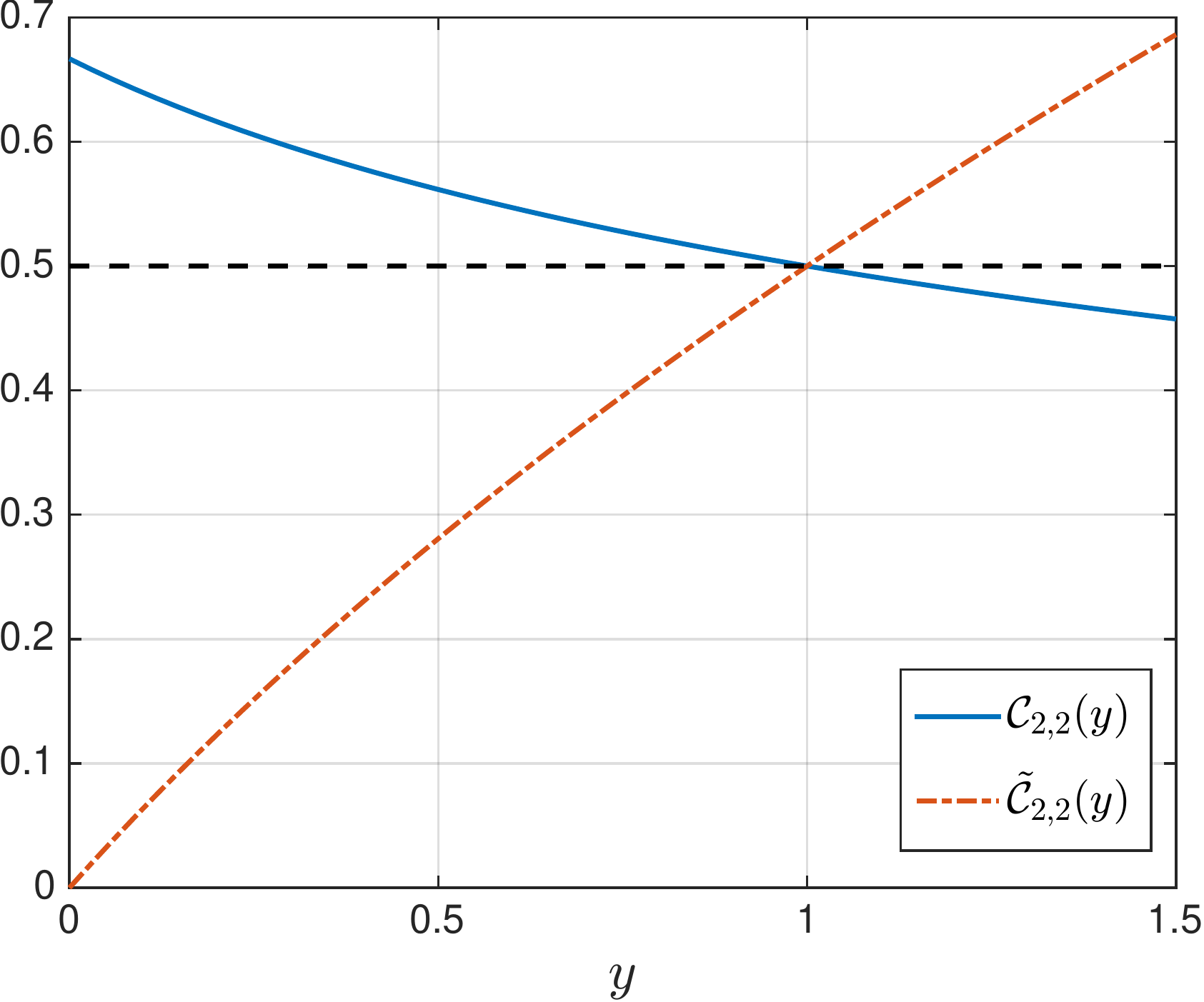}
    \caption{Functions $\sspcoef_{2,2}(y)$ and $\tildesspcoef_{2,2}(y)$ for the class of explicit
    two-step, second-order \class{\plmm}.
    The dotted line shows $\tildesspcoef_{\text{LMM}} = \sspcoef_{2,2}(1)$ for this particular class of
    methods.}
    \label{fig:DLMM}
\end{figure}

For purely hyperbolic problems the spatial discretizations are usually chosen in such a way that $\F$
and $\tildeF$ satisfy \eqref{eq:DLMMDiffFE} under the same step-size restriction.
However, in many other cases (e.g. advection-reaction problems) this is not the case, as shown in
Example~\ref{ex:DLMM_PDE}.
First, we mention the following lemma which is an extension of
\cite{Donat/Higueras/Martinez-Gavara:IMEXStiffReaction}; its proof can be found in
Appendix~\ref{appx:ProofLemmas}.
\begin{lemma}\label{lem:ALMMStepSize}
	Consider the function
	\begin{align*}
		f(u) = \sum_{i=1}^n f_i(u)
	\end{align*}
	and assume that there exist $\epsilon_i > 0 $ such that $||u + \tau f_i(u)|| \le ||u||$ for
	$0 \le \tau \le \epsilon_i$, $i \inset{1}{n}$, where $||\cdot||$ is a convex functional.
	Then $||u + \tau f(u)|| \le ||u||$ for $0 \le \tau \le \epsilon$, where
	\begin{align*}
		\epsilon = \left(\sum_{i=1}^n\frac{1}{\epsilon_i}\right)^{-1}.
	\end{align*}
\end{lemma}
\begin{example}\label{ex:DLMM_PDE}
	Consider the LeVeque and Yee problem
	\cite{LeVeque/Yee:1990:HCLStiffTerms,Donat/Higueras/Martinez-Gavara:IMEXStiffReaction}
	\begin{align*}
		U_t + f(U)_x = s(U), \quad U(x,0) = U_0(x), \quad x \in \R, \,  t \ge 0,
	\end{align*}
	where $s(U) = -\mu U (U - 1) (U - \frac{1}{2})$ and $\mu > 0$.
	Let $u_i(t) \approx U(x_i,t)$; then first-order upwind semi-discretization yields
	\begin{align*}
		u'(t) = F(u(t)) = D(u(t)) + S(u(t)), \quad u(0) = u_0, \quad t > 0,
	\end{align*}
	where
	\begin{align*}
		D_i(u) = -\frac{f(u_i) - f(u_{i-1})}{\Dx}, \qquad S_i(u) = s(u_i).
	\end{align*}
	Consider also the downwind discretizations
	\begin{align*}
		\widetilde{D}_i(u) = -\frac{f(u_{i+1}) - f(u_i)}{\Dx}, \qquad \widetilde{S}_i(u) = s(u_i),
	\end{align*}
	and let $\widetilde{F} = \widetilde{D} + \widetilde{S}$.
	If $u \in [0,1]$, it can be easily shown that
	\begin{align*}
		&0 \le u + \Dt\, S(u) \le 1 \quad \text{for } 0 \le \Dt \le \DtFE =
		\frac{2}{\mu}, \\
		&0 \le u - \Dt\, \widetilde{S}(u) \le 1 \quad \text{for } 0 \le \Dt \le \tildeDtFE =
		\frac{16}{\mu}.
	\end{align*}
	Using Lemma~\ref{lem:ALMMStepSize} we then have that
	\begin{align*}
		&0 \le u + \Dt\, F(u) \le 1 \quad \text{for } 0 \le \Dt \le \DtFE = 
		\frac{2\tau}{2 + \mu\tau}, \\
		&0 \le u - \Dt\, \widetilde{F}(u) \le 1 \quad \text{for } 0 \le \Dt \le \tildeDtFE =
		\frac{16\tau}{16 + \mu\tau},
	\end{align*}
	where $\tau > 0$ is such that
	\begin{align*}
		&0 \le u + \Dt\, D(u) \le 1 \quad \text{for } 0 \le \Dt \le \tau, \\
		&0 \le u - \Dt\, \widetilde{D}(u) \le 1 \quad \text{for } 0 \le \Dt \le \tau.
	\end{align*}
	Note that $\DtFE < \tildeDtFE$ for all positive values of $\mu$ and $\tau$.
	Therefore, under assumptions (1) and (2) above, the forward Euler step size must be equal to
	$2\tau/(2 + \mu\tau)$ so that the numerical solution is stable.
	Let $y = \DtFE/\tildeDtFE$, then for all $y < 1$ we have
	$\tildesspcoef_{\text{LMM}} = \sspcoef(1) < \sspcoef(y)$, hence not considering SSP \class{\plmm}
	will always result to a stricter step-size restriction.
	Suppose $\mu$ is relatively small so that the problem is not stiff and explicit methods could be
	used.
	For instance, among the class of explicit two-step, second-order LMMs, there is no classical
	SSP method and the optimal downwind method has SSP coefficient $\tildesspcoef_{\text{LMM}} = 1/2$.
	Let $\mu\tau = 2/3$, then the step-size bound for downwind SSP methods such that the solution remains
	in $[0,1]$ is $\Dt \le 0.375\tau$.
	Using the optimal two-step, second-order SSP \plmm larger step sizes are allowed since
	$\Dt \le  \sspcoef(y) \,\DtFE = 0.3929\tau$, where $y = \frac{16 + \mu\tau}{8(2 + \mu\tau)}$.
\end{example}

\section{Monotonicity of additive \class{\lmm}}\label{sec:ALMM}
Following the previous example, it is natural to study the monotonicity properties of additive methods
applied to problems which consist of components that describe different physical processes.
A $k$-step \almm for the solution of the initial value problem
\begin{equation}\label{eq:AddODE}
	\begin{split}
		\u'(t) &= \F(\u(t)) + \hatF(\u(t)), \quad t \ge t_0 \\
		\u(t_0) &= \u_0,
	\end{split}
\end{equation}
 takes the form
\begin{align}\label{eq:ALMM}
	\u_n = \sum_{j=0}^{k-1}\alpha_j\u_{n-k+j} +
	\Dt\sum_{j=0}^k\left(\beta_j\F(\u_{n-k+j}) + \hat{\beta}_j\hatF(\u_{n-k+j})\right).
\end{align}
The method is explicit if $\beta_k = \hat{\beta}_k = 0$.
It can be shown that method \eqref{eq:ALMM} is order $p$ accurate if
\begin{align}\label{eq:ALMMOrderCond}
\begin{gathered}
	\sum_{j=0}^{k-1} \alpha_j = 1, \qquad \sum_{j=0}^{k-1} j\alpha_j + \sum_{j=0}^k \beta_j = k, \qquad
	\sum_{j=0}^{k-1} j\alpha_j + \sum_{j=0}^k \hat{\beta_j} = k, \\
	\sum_{j=0}^{k-1}\alpha_j j^i + \sum_{j=0}^{k}\beta_j ij^{i-1} = k^i, \quad
	\sum_{j=0}^{k-1}\alpha_j j^i + \sum_{j=0}^{k}\hat{\beta}_j ij^{i-1} = k^i,\quad i \inset{2}{p}.
\end{gathered}
\end{align}
The operators $\F$ and $\hatF$ generally approximate different derivatives and also have different
stiffness properties.
We extend the analysis of monotonicity conditions for LMMs by assuming that $\F$ and $\hatF$
satisfy 
\begin{subequations}\label{eq:ALMMDiffFE}
	\begin{align}
		\|\u + \Dt\F(\u)\| \leq \|\u\|, &\quad \forall \u \in \R^m, \; 0 \leq \Dt \leq \DtFE,
		\label{eq:ALMMDiffFEa} \\
		\|\u + \Dt\hatF(\u)\| \leq \|\u\|, &\quad \forall \u \in \R^m, \; 0 \leq \Dt \leq \hatDtFE,
		\label{eq:ALMMDiffFEb}
	\end{align}
\end{subequations}
respectively.
\begin{definition}\label{def:SSPALMM}
	An \almm \eqref{eq:ALMM} is said to be strong stability preserving (SSP) if the following
	monotonicity conditions
	\begin{equation}\label{eq:ALMMMonCond}
		\begin{split}
			\beta_j, \hat{\beta}_j \geq 0, &\quad j \inset{0}{k}, \\
			\alpha_j - r\beta_j - \hat{r}\hat{\beta}_j \geq 0, &\quad j \inset{0}{k-1}.
		\end{split}
	\end{equation}
	hold for $r \ge 0$ and $\hat{r} \ge 0$.
	For a fixed $y = \hat{r}/r $ the method has SSP coefficients $(\sspcoef(y), \hatsspcoef(y))$, where
	\begin{align}\label{eq:ALMMSSPcoef}
		\sspcoef(y) = \sup \bigl\{ r \ge 0 : \text{monotonicity conditions
		\eqref{eq:ALMMMonCond} hold with } \hat{r} = y r \bigr\}
	\end{align}
	and $\hatsspcoef(y) = y \, \sspcoef(y)$.
\end{definition}
\noindent As in Section~\ref{sec:DLMM}, it is clear that whenever the set in \eqref{eq:ALMMSSPcoef} is empty then
the method is non-SSP; in such cases we say the method has SSP coefficient equal to zero.

Define the vectors $\bm{a}_j,\bm{b}_j(x) \in \R^{p+1}$ as in \eqref{eq:DLMMVectora} and
\eqref{eq:DLMMVectorb}.
Then using the substitution
\begin{align}\label{eq:alpha_j}
	\gamma_j = \alpha_j - r\beta_j - \hat{r}\hat{\beta}_j \quad \text{for} \quad j \inset{0}{k-1},
\end{align}
the order conditions \eqref{eq:ALMMOrderCond} can be expressed in terms of vectors
$\bm{a}_j,\bm{ b}_j(x) $:
\begin{subequations}\label{eq:ALMMOrderCondVectorForm}
	\begin{align}
		\sum_{j=0}^{k-1} (\gamma_j + \hat{r}\hat{\beta}_j)\bm{a}_j + 
		\sum_{j=0}^k r\beta_j \bm{b}_j(r^{-1}) = \bm{a}_k,
		\label{eq:ALMMOrderCondVectorForma} \\
		\sum_{j=0}^{k-1} (\gamma_j + r\beta_j) \bm{a}_j + 
		\sum_{j=0}^k \hat{r}\hat{\beta}_j \bm{b}_j(\hat{r}^{-1}) = \bm{a}_k.
		\label{eq:ALMMOrderCondVectorFormb}
	\end{align}
\end{subequations}
The above equations suggest a change of variables.
Instead of considering the method's coefficients in terms of the column vectors
\begin{align*}
	\bm{\alpha} = (\alpha_0, \dots, \alpha_{k-1})^\intercal, \quad
	\bm{\beta} = (\beta_0, \dots, \beta_k)^\intercal, \quad
	\bm{\hat{\beta}} = (\hat{\beta}_0, \dots, \hat{\beta}_k)^\intercal,
\end{align*}
and the order conditions independent of $r$ and $\hat{r}$, one can consider the coefficients 
$\bm{\gamma}, \bm{\beta}, \bm{\hat{\beta}}$ under the substitution \eqref{eq:alpha_j}.
Let $\hat{r} = yr$. Then the order conditions can be written as functions of $r$.
In particular the system of $p+1$ equations \eqref{eq:ALMMOrderCondVectorForma} can be written as
$A(r)\bm{x}(r) = \bm{a}_k$, 
where
\begin{align*}
	A(r) = 
	\def\arraystretch{2.5}
	\left[\begin{array}{c|c|c|c|c|c|c}
		\bm{a}_0 & \dots & \bm{a}_{k-1} & r\bm{b}_0(r^{-1}) & \dots & r\bm{b}_{k-1}(r^{-1}) &  
		 r\bm{b}_k(r^{-1}) \\[1em]
	\end{array}\right]
\end{align*}
and $\bm{x}(r) = \bigl(\bm{\delta}(r),  \bm{\beta}\bigr)^\intercal \in \R^{2k+1}$ with
$\delta_j(r) = \gamma_j + y r \hat{\beta}_j$, $j \inset{0}{k-1}$.
Define the feasible set
\begin{align}\label{eq:ALMMFeasibilityProgram}
	P(r) = \{\bm{x} \in \R^{2k+1}: A(r)\bm{x}(r) = \bm{a}_k, \; \bm{x}(r) \ge 0 \}.
\end{align}
For a given $y$, if there exists a $k$-step, $p$-order accurate SSP \almm \eqref{eq:ALMM}
with SSP coefficient $\sspcoef(y)$, then $P\bigl(\sspcoef(y)\bigr)$ is non-empty.

Since we would like to obtain the method with the largest possible SSP coefficient, then for a fixed 
$k \ge 1$, $p \ge 1$ and a given $y$, we define
\ifpaper
\else
	\begin{align*}
		\sspcoef_{k,p}(y) = \sup_{\bm{\alpha},\bm{\beta},\bm{\hat{\beta}}} \bigl\{ \sspcoef(y) > 0 :
		\sspcoef(y)
		\text{ is the SSP coefficient of a $k$-step method \eqref{eq:ALMM} of order $p$} \bigr\}.
	\end{align*}
\fi
\begin{definition}\label{def:ALMMOptimal}
	Given $y$, an SSP  $k$-step \almm \eqref{eq:ALMMMonCond} of order $p$ is called
	\emph{optimal} if the order conditions \eqref{eq:ALMMOrderCond} are satisfied and
	$\sspcoef(y) = \sspcoef_{k,p}(y)$.
\end{definition}
\begin{theorem}\label{thm:ALMMNonzeroCoeff}
	Let $k \ge 1$, $p \ge 1$ be given such that $0 < \sspcoef_{k,p}(y) < \infty$ for a given $y$.
	Then there exists a $k$-step optimal SSP \almm \eqref{eq:ALMM} of order $p$ with at most $p$
	non-zero coefficients $\delta_j, \beta_i$, where $\delta_j = \alpha_j - \sspcoef_{k,p}(y)\beta_j$,
	$j \inset{0}{k-1}$ and $i \inset{0}{k}$.
\end{theorem}
\begin{proof}
	Let $k \ge 1$, $p \ge 1$ and $y$ be given.
	Consider an optimal $k$-step SSP \almm \eqref{eq:ALMM} of order $p$ with SSP coefficient
	$\sspcoef_{k,p}(y) > 0$.
	Define $\gamma_j = \alpha_j - \sspcoef_{k,p}(y)\beta_j - \hat{\sspcoef}_{k,p}(y)\hat{\beta}_j$ and 
	$\delta_j = \gamma_j + \hat{\sspcoef}_{k,p}(y)\hat{\beta}_j$ for  $j \inset{0}{k-1}$.
	Then the vector $\bm{x} = (\bm{\delta}, \bm{\beta})^\intercal \in \R^{2K+1}$ belongs to the feasible
	set \eqref{eq:ALMMFeasibilityProgram} when $r = \sspcoef_{k,p}(y)$.
	
	Suppose $\bm{x}$ has at least $p+1$ non-zero coefficients and	 let $S$ be the set of columns of the
	matrix $A(r)$ in \eqref{eq:ALMMFeasibilityProgram} corresponding to the non-zero elements of $x$.
	We distinguish two cases.
	First, assume that the set $S$ does not span $\R^{p+1}$.
	Then, similarly to the proof of Theorem~\ref{thm:DLMMNonzeroCoeff}, $x$ consists of at most $p$
	non-zero elements.
	If now $S$ spans $\R^{p+1}$, let	$\bm{x}_p = (\bm{x}_B, \bm{x}_N)^\intercal$ be a permutation of
	$\bm{x}$ such that $\bm{x}_B^\intercal \in \R^{p+1}$ is a strictly positive vector and
	$\bm{x}_N^\intercal \in \R^{2k- p}$ is non-negative.
	We can permute the columns of $A(r)$ in \eqref{eq:ALMMFeasibilityProgram} in the same way, yielding
	$A_p(r) = [B(r) \;|\; N(r)]$, where $B \in \R^{(p+1) \times (p+1)}$ and
	$N \in \R^{(p+1) \times (2k- p)}$.
	Again, following the reasoning of the proof of Theorem~\ref{thm:DLMMNonzeroCoeff}, there exists
	$\epsilon > 0$ such that
	$\bm{x}_p^* = (\bm{x}_B(\sspcoef_{k,p}(y) + \epsilon), \bm{x}_N)^\intercal$ is a permutation of
	$\bm{x}^* = (\bm{\delta}^*, \bm{\beta}^*)^\intercal$ that solves
	$A(\sspcoef_{k,p}(y) + \epsilon) \bm{x} = \bm{a}_k$.
	
	Moreover, for each index $j$ in $\bm{x}^*$ such that $\delta^*_j > 0$, we can choose $\gamma^*_j$ so
	that $\beta^*_j = \hat{\beta}^*_j$.
	Then, $\bm{x}^*$ satisfies \eqref{eq:ALMMOrderCondVectorFormb} as well.
	But this contradicts to the optimality of the method since we have constructed a $k$-step
	SSP \almm of order $p$ with coefficients given by $\bm{x}^*$ and SSP coefficient
	$\sspcoef_{k,p}(y) + \epsilon$.
\end{proof}
\begin{lemma}\label{lem:BetasEquality}
	For a given $k \ge 1, p \ge 1$ an optimal \almm \eqref{eq:ALMM} has $\beta_j = \hat{\beta}_j$
	for all $j \inset{0}{k}$.
\end{lemma}
\begin{proof}
	Consider an optimal method \eqref{eq:ALMM} of order $p$.
	From Theorem~\ref{thm:ALMMNonzeroCoeff} at most $p$ coefficients $\delta_j, \beta_i$,
	$j \inset{0}{k-1}$, $i \inset{0}{k}$ are non-zero.
	Let $\bm{v} = \bm{\beta} - \bm{\hat{\beta}}$, then $\bm{v}$ has at most $p$ non-zero elements.
	Subtracting the order conditions \eqref{eq:ALMMOrderCondVectorForm} results in
	\begin{align*}
		\sum_{i \in I} v_i  \bm{\bar{a}}_i = 0,
	\end{align*}
	where $I$ is the set of distinct indices for which $v_i$'s are non-zero.
	The vectors $\bm{\bar{a}}_i = (1, i, \dots, \allowbreak i^{p-1})^\intercal$, $i \in I$ are linearly
	independent (see \cite[Chapter~21]{Higham:2006:Accuracy&StabilityNumAlg}), therefore $\bm{v}$ must
	be identically equal to zero.
	Hence, $\beta_j = \hat{\beta}_j$ for all $j \inset{0}{k}$.
\end{proof}
The main result of this section relies on Theorem~\ref{thm:ALMMNonzeroCoeff} and
Lemma~\ref{lem:BetasEquality}.
\begin{theorem}\label{thm:OptimalALMM}
	For a given $k \ge 1, p \ge 1$ an optimal \almm with SSP coefficient $\sspcoef_{k,p}$ and
	corresponding SSP coefficient $\hatsspcoef_{k,p}$ is equivalent to the optimal $k$-step optimal SSP
	LMM \eqref{eq:LMM} of order $p$ with SSP coefficient $\sspcoef_{k,p} + \hatsspcoef_{k,p}$.
\end{theorem}
\begin{proof}
	Consider an optimal method \eqref{eq:ALMM} of order $p$ with SSP coefficient $\sspcoef_{k,p}$ and
	$\hatsspcoef_{k,p} = y\,\sspcoef_{k,p}$ for some $y \in [0,\infty)$.
	From Lemma~\ref{lem:BetasEquality} we have $\beta_j = \hat{\beta}_j$ for all $j$, therefore
	monotonicity conditions \eqref{eq:ALMMMonCond} yield
	$\min_j \frac{\alpha_j}{\beta_j} = \sspcoef_{k,p} + \hatsspcoef_{k,p}$.
	Thus the \almm is equivalent to the optimal $k$-step SSP LMM method of order $p$ with SSP coefficient
	$\sspcoef_{k,p} + \hatsspcoef_{k,p}$.
\end{proof}

\subsection{Monotone IMEX \class{\lmm}}
Based on Theorem~\ref{thm:OptimalALMM}, it is only interesting to consider Implicit-Explicit (IMEX) SSP
\class{\lmm}.
Such methods are particularly useful for initial value problems \eqref{eq:ALMM} where $\F$ represents a
non-stiff or mild stiff part of the problem, and $\hatF$ a stiff term for which implicit integration is
required.
The following theorem provides sufficient conditions for monotonicity for the numerical solution
of an IMEX method.
\begin{theorem}\label{thm:IMEXMonotonicity}
	Consider the additive problem \eqref{eq:AddODE} for which $\F$ and $\hatF$ satisfy 
	\eqref{eq:ALMMDiffFE}, for some  $\DtFE > 0$ and  $\hatDtFE > 0$.
	Let an IMEX LMM \eqref{eq:ALMM} with coefficients $\beta_k = 0$, $\hat{\beta_k} \ne 0$ be
	strong-stability-preserving with SSP coefficients $(\sspcoef(y),\hatsspcoef(y))$ for
	$y = \DtFE/\hatDtFE$.
	Then, the numerical solution satisfies the monotonicity condition \eqref{eq:Monotonicity} under a
	step-size restriction $\Dt \leq \min\{\sspcoef \,\DtFE,\hatsspcoef \, \hatDtFE\}$.
\end{theorem}
\begin{proof}
	The proof is similar to that of Theorem~\ref{thm:DLMMMonotonicity}.
\end{proof}

As in the Section~\ref{sec:DLMM}, the minimum step size in \eqref{thm:IMEXMonotonicity} occurs when
$\sspcoef \,\DtFE = \hatsspcoef \, \hatDtFE$.
For a given $k \ge 1$ and $p \ge 1$, we would like to find the largest possible value $\sspcoef_{k,p}(y)$
such that an optimal IMEX method is SSP with coefficients
$(\sspcoef_{k,p},\sspcoef_{k,p} \,\DtFE/\hatDtFE)$.
Setting $y := \DtFE/\hatDtFE$ and combining the inequalities \eqref{eq:ALMMMonCond} and the
order conditions \eqref{eq:ALMMOrderCond} we can form the following optimization problem:
\begin{gather}
	\max_{\{\bm{\gamma},\bm{\beta},\bm{\hat{\beta}},r\}} \; r, \qquad \text{subject to}
	\nonumber\\[10pt]
	\left\{\arraycolsep=1.4pt\def\arraystretch{2.2}
	\begin{array}{ll}\label{eq:IMEXOptPrb}
		\displaystyle\sum_{j=0}^{k-1} \gamma_j + r(\beta_j + y\hat{\beta}_j) = 1, \quad
		\displaystyle\sum_{j=0}^{k-1}\bigl(\gamma_j + r(\beta_j + y\hat{\beta}_j)\bigr) j +
		\beta_j = k, & \\
		\displaystyle\sum_{j=0}^{k-1}\bigl(\gamma_j + r(\beta_j + y\hat{\beta}_j)\bigr) j^i +
		\beta_j ij^{i-1} = k^i, & i \inset{2}{p}, \\
		\displaystyle\sum_{j=0}^{k-1} (\beta_j - \hat{\beta}_j) - \hat{\beta}_k = 0, \quad
		\displaystyle\sum_{j=0}^{k-1} (\beta_j - \hat{\beta}_j) j^i - \hat{\beta}_k k^i = 0,
		& i \inset{1}{p-1}, \\
		\gamma_j \ge 0, \beta_j \ge 0, & j \inset{0}{k-1}, \\
		\hat{\beta}_j \ge 0, & j \inset{0}{k}, \\
		r \ge 0.
    \end{array}\right.
\end{gather}
By using bisection in $r$, the optimization problem \eqref{eq:IMEXOptPrb} can be viewed as a sequence of
linear feasible problems, as suggested in \cite{Ketcheson:2009:OptimalMonotonicityGLM}.
We solved the above problem using \texttt{linprog} in Matlab and found optimal IMEX SSP methods for
$k \inset{1}{40}, p \inset{1}{15}$ and for different values of $y$.
Similarly to \class{\ark} \cite{Higueras:2006:ARK}, we can define the feasibility SSP region of IMEX SSP
methods for a fixed $k \ge 1$ and $p \ge 1$ by
\begin{align*}
	R_{k,p} = \bigl\{ (r, \hat{r}) : y \in \R^+ \text{ and monotonicity conditions \eqref{eq:ALMMMonCond}
	hold for } r \ge 0 , \hat{r}= ry \bigr\}.
\end{align*}
For instance, the feasibility SSP regions for three-step, second-order and six-step, fourth-order IMEX
methods are shown in Figure~\ref{fig:IMEX}.

\begin{figure}[!ht]
	\centering
	\subfigure[Three-step, second-order IMEX SSP region]
	{\includegraphics[width=0.48\textwidth]{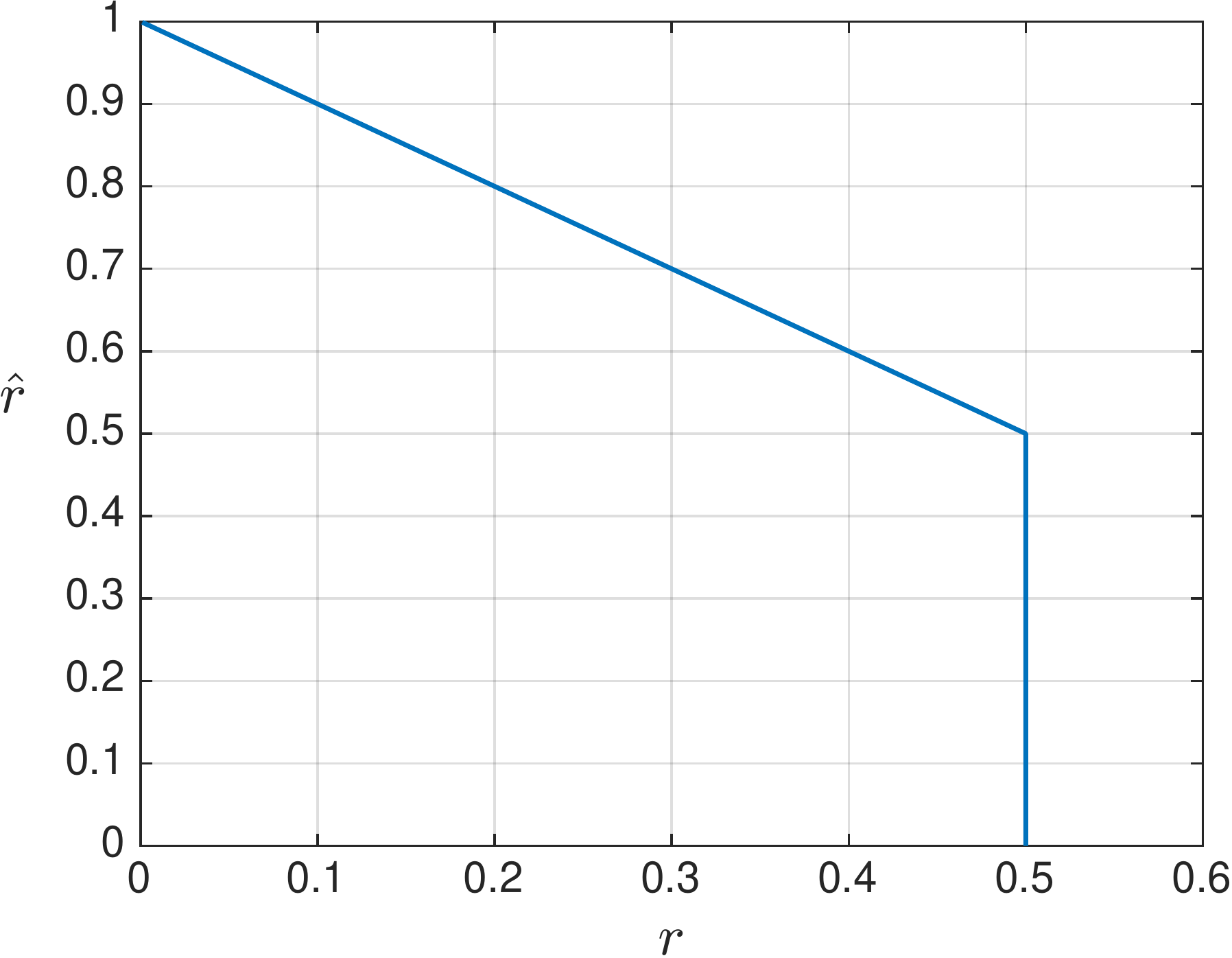}}\label{fig:IMEXa}
	\quad
	\subfigure[Six-step, fourth-order IMEX SSP region]
	{\includegraphics[width=0.48\textwidth]{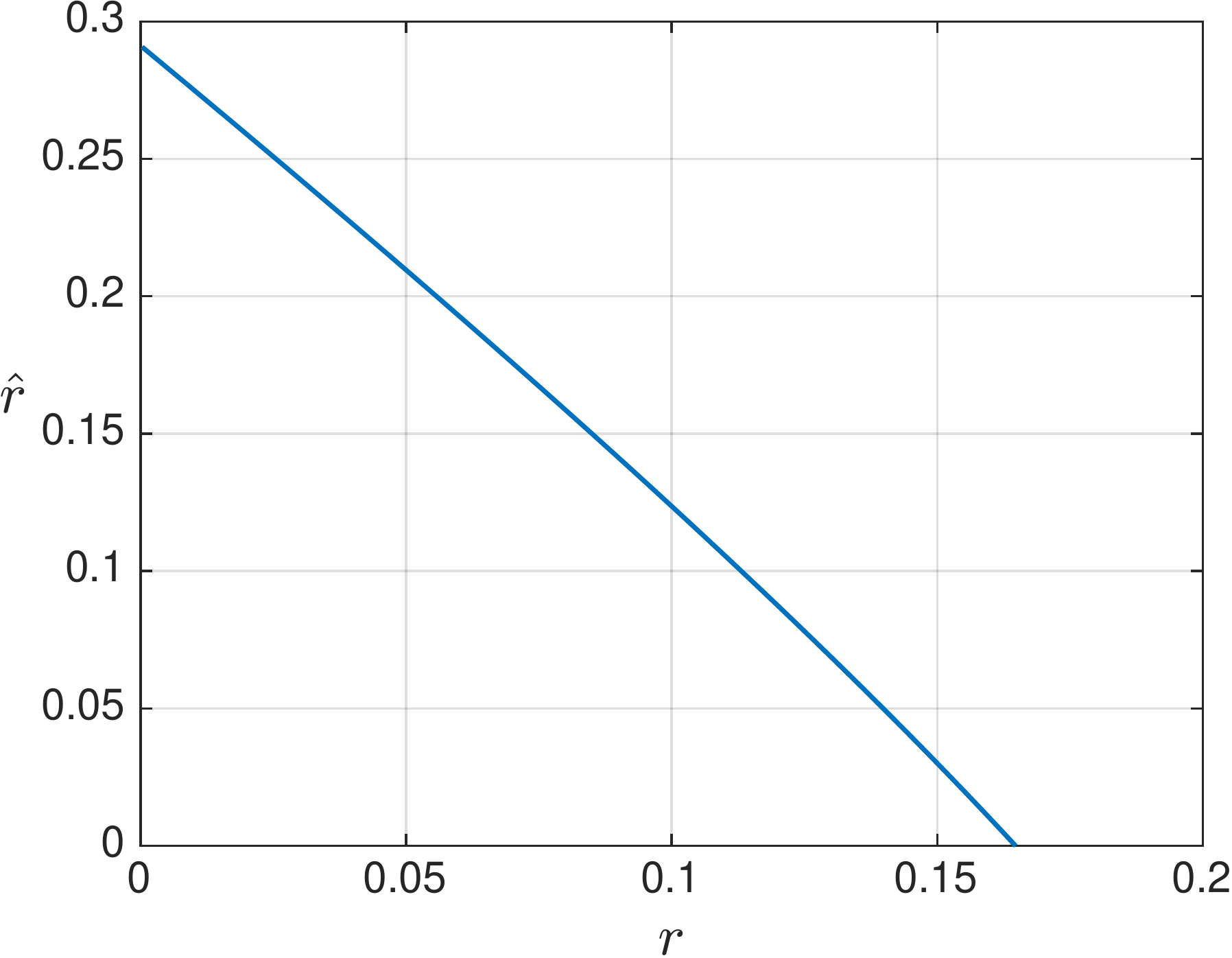}}\label{fig:IMEXb}
	\caption{SSP regions of IMEX LMMs.
	Consider the half-line starting from the origin with a slope $y$ with the $r$-axis.
	Then, the intersection of the line $\hat{r} = y r$, $y >0$ with the boundary of the SSP region
	corresponds to an optimal IMEX LMM with SSP coefficient $\sspcoef(y)$.}
    \label{fig:IMEX}
\end{figure}

As mentioned in \cite[Section~2.1]{Hundsdorfer/Ruuth:2007:IMEX-LMM} the SSP coefficients of IMEX SSP
methods in the case the forward Euler ratio $y = \DtFE/\hatDtFE$ is equal to one are not large.
The same seems to hold when considering SSP IMEX methods for additive problems \eqref{eq:AddODE}
satisfying \eqref{eq:ALMMDiffFE} for any values $y \ge 0$ (see Figure~\ref{fig:IMEX}).
Thus, instead of requiring both parts of an IMEX method to be SSP, one can impose SSP conditions only on
the explicit part and optimize stability properties for the implicit method.
Second order methods among this class of methods have been studied in \cite{Gjesdal:2007:2nd0rderIMEX},
whereas in \cite{Ruuth/Hundsdorfer:2005:GeneralMonotonicityBoundednessLMM} higher order IMEX methods with
optimized stability features were constructed based on general monotonicity and boundedness properties
of the explicit component.

\section{Conclusion and future work}\label{sec:Conclusion}
We have investigated a generalization of the \class{\lmm} with upwind- and downwind-biased operators
introduced in \cite{Shu:1988:TVD,Shu/Osher:1988:ENO}, by considering problems in which the downwind
operator satisfies a forward Euler condition with different step-size restriction than that of the upwind
operator.
We expressed the \class{\plmm} in an additive form and analyzed their monotonicity
properties.  By optimizing in terms of the upwind and downwind Euler step sizes,
methods with larger SSP step sizes are obtained for such problems.
We studied additive problems in the same framework, and we have shown that
when both parts of the method are explicit (or both parts are implicit),
the optimal additive SSP methods lie within the class of traditional
(non-additive) SSP linear multistep methods.
Finally, we have seen that IMEX SSP methods for additive problems allow 
relatively small monotonicity-preserving step sizes.

The concepts of additive splitting and downwind semi-discretization can be
combined to yield downwind IMEX LMMs of the form (applying downwinding
to the non-stiff term):
\ifpaper
\else
	\begin{align}\label{eq:IMEX}
		\u_n = \sum_{j=0}^{k-1}\alpha_j\u_{n-k+j} +
		\Dt\sum_{j=0}^{k-1}\left(\beta_j\F(\u_{n-k+j}) - \tilde{\beta}_j\tildeF(\u_{n-k+j})\right) +
		\Dt\sum_{j=0}^k\hat{\beta}_j\hatF(\u_{n-k+j}),
	\end{align}
\fi
where $\F$ and $\tildeF$ satisfy the forward Euler conditions \eqref{eq:DLMMDiffFE} and the explicit part
is an SSP \plmm\unskip.
Preliminary results show that it is possible to obtain second order IMEX \class{\lmm} with two or three
steps, where the implicit part is A-stable and the explicit part is an optimal SSP \plmm\unskip.
This generalization allows the construction of new IMEX methods with fewer steps for a given order of
accuracy and with larger SSP coefficients (for the explicit component).
Moreover, the best possible IMEX method can be chosen based on the ratio of forward Euler step sizes of
the non-stiff term in \eqref{eq:AddODE}.
Also, it is worth investigating the possibility of obtaining A($\alpha$)-stable implicit parts
whenever A-stability is not feasible.
Work on optimizing the stability properties of the IMEX methods \eqref{eq:IMEX} is ongoing and will be
presented in a future work.
Analysis of SSP \class{\plmm} with variable step sizes and monotonicity properties of \class{\plmm} with
special starting procedures can also be studied.

\subsection*{Acknowledgment}
The authors would like to thank Lajos L{\'o}czi and Inmaculada Higueras for carefully reading this paper
and making valuable suggestions and comments.

\section*{Appendix}
\appendix
\section{Proofs of Lemmata in Section~\ref{sec:DLMM}}\label{appx:ProofLemmas}
In this section we present the proofs of some technical lemmata that were omitted in the previous
sections.
\begin{proof}[Proof of Lemma~\ref{lem:DLMMConvexity}]
	Consider a set of distinct vectors $S = \{\bm{x}_1, \dots, \bm{x}_m\}$ in $\R^n$.
	Let a non-zero vector $\bm{y} \in C$ be given.
	Then there exist non-negative coefficients $\lambda_j$ that sum to unity such that
	\begin{align*}
		\bm{y} = \sum_{j=1}^{m} \lambda_j x_j.
	\end{align*}
	If $\bm{x}_1, \dots, \bm{x}_m$ are linearly independent, it must be that $m \le n$ and both parts (a)
	and (b) of the lemma hold trivially.
	Therefore, assume the vectors in $S$ are linearly dependent.
	Then, we can find $\mu_j$ not all zero and at least one which is positive, such that
	\begin{align*}
		\sum_{j=1}^{m} \mu_j \bm{x}_j = 0.
	\end{align*}
	Define
	\begin{align*}
		\nu = \min_{1 \le j \le m}\left\{\frac{\lambda_j}{\mu_j} : \mu_j > 0\right\} =
			\frac{\lambda_{j_0}}{\mu_{j_0}};
	\end{align*}
	then we have $\nu\mu_j \le \lambda_j$ for all $j \inset{1}{m}$, where equality holds for at least
	$j= j_0$.
	Let $\tilde{\lambda}_j = \lambda_j - \nu\mu_j$ for $j \inset{1}{m}$.
	By the choice of $\nu$, all coefficients $\tilde{\lambda}_j$ are non-negative and at least one of
	them is equal to zero.
	Note that
	\begin{align*}
		\bm{y} = \sum_{j=1}^{m} \lambda_j \bm{x}_j - \nu\sum_{j=1}^{m} \mu_j \bm{x}_j =
			\sum_{j=1}^{m} \tilde{\lambda}_j\bm{x}_j,
	\end{align*}
	hence $\bm{y}$ can be expressed as a non-negative linear combination of at most $m-1$ vectors in $S$.
	The above argument can be repeated until $\bm{y}$ is written as a non-negative linear combination of
	$\bm{x}_1, \dots, \bm{x}_r$ linearly independent vectors, where $r \le n$.
	This proves part (a).

	For part (b), suppose $\bm{x}_1, \dots, \bm{x}_m$ are linearly dependent and belong in
	$\{(1,\bm{v}) : \bm{v} \in \R^{n-1}\}$.
	Then, any non-zero vector $\bm{y} \in C$ has the form $(1,\bm{v})^\intercal$, $\bm{v} \in \R^{n-1}$
	and from part (a) can be written as a non-negative combination of at most $n$ linearly independent
	vectors in $S$ with coefficients $\tilde{\lambda}_j$.
	In addition $\sum_{j=1}^{m} \tilde{\lambda}_j = 1$, since the first component of all $\bm{x}_j$ and
	$\bm{y}$ is one.
\end{proof}
\begin{proof}[Proof of Lemma~\ref{lem:ALMMStepSize}]
	Let $p_i(u;\epsilon_i) := u + \epsilon_i f_i(u)$, then we have
	\begin{align*}
		f_i(u) = \frac{p_i(u;\epsilon_i) - u}{\epsilon_i}, \quad \text{for } i \inset{1}{n}.
	\end{align*}
	Using $\sum_{i=1}^n \epsilon/\epsilon_i = 1$ and the assumption of the lemma, it can be shown that
	\begin{align*}
		\|u + \epsilon f(u)\|
		&= \left\|u + \sum_{i=1}^n\frac{\epsilon}{\epsilon_i}\left(p_i(u;\epsilon_i)-u\right)\right\| \\
		&= \left\|\sum_{i=1}^n\frac{\epsilon}{\epsilon_i}p_i(u;\epsilon_i)\right\| \\
		&\le \sum_{i=1}^n\frac{\epsilon}{\epsilon_i}\|u\| = \|u\|.
	\end{align*}
	The rest of the proof relies on \cite[Lemma~II.5.1]{Martin:1976:OperatorsAndDiffEqnsBanachSpaces}.
	If $0 \le \tau < \epsilon$, then there exist $0< \rho < 1$ such that $\tau = (1 - \rho)\epsilon$.
	Then
	$u + \tau f(u) = u + (1 - \rho)\epsilon f(u) = \rho u + (1 - \rho)\left(u + \epsilon f(u)\right)$
	and hence
	\begin{align*}
		\|u + \tau f(u)\| - \|u\|&\le \rho\|u\| + (1 - \rho)\|u + \epsilon f(u)\| - \|u\| \\
		&= \left(1 - \rho\right)\left(\|u + \epsilon f(u)\| - \|u\|\right) \\
		&\le \|u + \epsilon f(u)\| - \|u\|.
	\end{align*}
	This implies that $\|u + \tau f(u)\| \le \|u + \epsilon f(u)\|$, whenever $0 \le \tau \le \epsilon$
	and the result follows.
\end{proof}

\ifpaper
\else

	\sloppy
	\printbibliography

\fi

\end{document}